\documentclass[12pt,a4paper,dvipdfm]{article}
\usepackage[margin=1 in]{geometry}
\usepackage{graphicx}
\usepackage{amssymb, amsmath}
\usepackage{amsthm}
\usepackage{dsfont}
\usepackage{enumerate}
\usepackage{setspace}
\begin{document}
\newcommand{\up}{\vspace*{-0.05cm}}
\allowdisplaybreaks[1]
\numberwithin{equation}{section}
\renewcommand{\theequation}{\thesection.\arabic{equation}}
\newtheorem{thm}{Theorem}[section]
\newtheorem{lemma}{Lemma}[section]
\newtheorem{pro}{Proposition}[section]
\newtheorem{prob}{Problem}[section]
\newtheorem{quest}{Question}[section]
\newtheorem{example}{Example}[section]
\newtheorem{cor}{Corollary}[section]
\newtheorem{conj}{Conjecture}[section]
\newtheorem{cl}{Claim}[section]
\newtheorem{df}{Definition}[section]
\newtheorem{rem}{Remark}[section]
\newtheorem{note}{Note}[section]
\newcommand{\beq}{\begin{equation}}
\newcommand{\eeq}{\end{equation}}
\newcommand{\<}[1]{\left\langle{#1}\right\rangle}
\newcommand{\be}{\beta}
\newcommand{\ee}{\end{enumerate}}
\newcommand{\Bul}{\mbox{$\bullet$ } }
\newcommand{\al}{\alpha}
\newcommand{\ep}{\epsilon}
\newcommand{\si}{\sigma}
\newcommand{\om}{\omega}
\newcommand{\la}{\lambda}
\newcommand{\La}{\Lambda}
\newcommand{\Ga}{\Gamma}
\newcommand{\ga}{\gamma}
\newcommand{\im}{\Rightarrow}
\newcommand{\2}{\vspace{.2cm}}
\newcommand{\es}{\emptyset}

\vspace{-2cm}
\markboth{R. Rajkumar and P. Devi}{Permutability Graphs of Non-abelian Groups}
\title{\LARGE\bf Permutability graphs of subgroups of some finite non-abelian groups}
\author{R. Rajkumar\footnote{e-mail: {\tt rrajmaths@yahoo.co.in}},\ \ \
P. Devi\footnote{e-mail: {\tt pdevigri@gmail.com}}\\
{\footnotesize Department of Mathematics, The Gandhigram Rural Institute -- Deemed University,}\\ \footnotesize{Gandhigram -- 624 302, Tamil Nadu, India}\\[3mm]
}
\date{}
\maketitle
\doublespacing

\begin{abstract}
In this paper, we study  the structure of the permutability graphs of subgroups, and the permutability graphs of non-normal subgroups of the following groups: the dihedral groups $D_n$, the generalized quaternion groups $Q_n$, the quasi-dihedral groups $QD_{2^n}$ and the modular groups $M_{p^n}$. Further,   we investigate the number of edges, degrees of the vertices, independence number, dominating number,  clique number, chromatic number, weakly perfectness, Eulerianness, Hamiltonicity of these graphs.
\end{abstract}
\section{Introduction}
\label{sec:intro}
One can study the properties of an algebraic structure by associating a suitable graph with it and by using of the tools of graph theory. 
In recent years this has been a topic of interest among algebraic graph theorists  and  they have contributed significantly; in particular, when the algebraic structure
is a group (see, for example \cite{abdolla}, \cite{abd},  \cite{bertram}, \cite{will}).

M. Bianchi, A. Gillio and L. Verardi in \cite{binachi 1}, defined a graph corresponding to a group $G$, called
the \emph{permutability graph of non-normal subgroups of $G$} having all the proper non-normal subgroups of $G$ as its vertices and two
vertices $H$ and $K$ are adjacent if  $HK=HK$; equivalently $HK$ is a subgroup of $G$. We denote this graph by $\Gamma_N(G)$.  They focused on finding the number of connected components
and the diameter of this graph. Further results  on these graphs can be found in \cite{gillio}, \cite{binachi 2}.

 In \cite{raj1}, the authors considered the general setting as follows: For a group $G$, \emph{the permutability
graph of subgroups of $G$}, denoted by $\Gamma(G)$, is a graph with vertex set consists of all the proper subgroups of $G$ and
two vertices in $\Gamma(G)$ are adjacent if  the two corresponding subgroups permute in $G$. There in we have studied the planarity of these graphs. Further properties of these graphs like bipartiteness, completeness and freeness of
these graphs from some class of graphs were investigated in \cite{raj2}.

The aim of this paper is to study the structure and properties of the permutability graphs of subgroups, and the permutability graphs of non-normal subgroups of finite non-abelian groups. Especially, we consider the dihedral groups $D_n$, the generalized quaternion groups $Q_n$, the quasi-dihedral groups $QD_{2^n}$ and the modular groups $M_{p^n}$. Even though the subgroup structure of these groups are well known, we particulary focus on how much information about the  subgroup permutability of these groups can be expressed in terms of the graph theoretic properties of their corresponding permutability graphs.

The rest of this paper is arranged as follows: In Section~\ref{sec:pre}, we introduce some basic definitions and notations that we will use in this article. In Section~\ref{sec:basic}, for a given group $G$, we present some basic results which gives the relationship between $\Gamma_N(G)$
and $\Gamma(G)$. In Section~\ref{sec:dih}, we consider the dihedral groups $D_n$ and study the structure and properties of $\Gamma_N(D_n)$
and $\Gamma(D_n)$. In particular, we
 give the degrees of  the vertices,   number of edges, independence number, dominating number,
chromatic number, clique number, weakly perfectness, Eulerianness of these two graphs. Also we investigate Hamiltonicity  of $\Gamma_N(D_n)$ and $\Gamma(D_n)$ for some values of $n$.
In Sections~\ref{sec:qua},~\ref{sec:quasi} and~\ref{sec:mod} we investigate the same  for the generalized quaternion groups $Q_n$,
the quasi-dihedral groups $QD_{2^n}$ and the modular groups $M_{p^n}$ respectively. In the last section, we conclude with some problems of further research.

\section{Preliminaries and Notations}\label{sec:pre}
For a simple graph $G$, we denote its vertex set and edge set by $V(G)$ and $E(G)$ respectively. A graph is \emph{complete} if all the
vertices are adjacent with each other. A complete graph on $n$ vertices is denoted by $K_n$. An \emph{independent set} is a set of
vertices in $G$ of which no two are adjacent. An independent set is \emph{maximal} if it is not a proper subset of any independent set of G. The \emph{independence number} $\alpha (G)$
of $G$ is the cardinality of a largest maximal independent set in $G$. A set $D$ of vertices in $G$ is a \emph{dominating set} in $G$ if  every vertex in  $G$ is
contained in or is adjacent to a vertex in $D$. The \emph{dominating number} $\gamma (G)$
of $G$ is the cardinality of a smallest dominating set in $G$. A \emph{clique}  is a set of vertices in $G$ such that any two
are adjacent. The \emph{clique number} $\omega (G)$ of $G$ is the cardinality of a largest  clique in $G$. The \emph{chromatic number}
$\chi (G)$ of $G$ is the smallest number of colours needed to
colour the vertices of $G$ such that no two adjacent vertices gets the same colour. A graph $G$ is said to be \emph{weakly perfect}
if $\omega (G)=\chi (G)$.
The $\deg_G(v)$ of a vertex $v$ of a graph $G$ is the number of edges incident with $v$. A graph $G$ is said to be \emph{$r$-partite} if the vertex of $G$ can be partitioned into $r$ sets such that no two vertices in each partition are adjacent. $G$ is \emph{complete $r$-partite} if every vertex in each partition is adjacent with all the vertices in the remaining partition.

A \textit{walk} joining two vertices $v_0$ and $v_n$ in $G$ is
an alternating sequence of vertices and edges $v_0$, $e_1$, $v_1$, $e_2$, $\ldots$ ,$v_{n-1}$, $e_n$, $v_n$ beginning and ending with
vertices such that each edge $e_i$ is incident with $v_{i-1}$ and $v_i$. A walk is a \textit{path} if all its vertices are distinct. We denote a path
joining two vertices $u$ and $v$ in $G$ simply, as $( u= ) v_0 - v_1 - v_2 - \cdots - v_n ( =v ) $ with the understanding that there is
an edge joining $v_{i-1}$ and $v_i$, for each $i$, $1 \leq i \leq n$. A walk is called \textit{closed} if its initial and terminal vertices coincides.
A closed walk in which all
the vertices are distinct is a \textit{cycle}. $G$ is said to be \emph{connected} if
any distinct two vertices are joined by a path. A \emph{component} of a graph $G$ is a maximal connected subgraph of $G$. A number of component of
a graph $G$ is denoted by $c(G)$.

Let $G_1 =(V_1, E_1)$ and $G_2 = (V_2, E_2)$ be two simple graphs. Their \textit{union} $G_1 \cup G_2$ is
a graph with
vertex set $V_1 \cup V_2$ and edge set $E_1 \cup E_2$. Their \textit{join} $G_1 + G_2$ is a graph consist of $G_1 \cup G_2$ together with all the lines
joining points of $V_1$ to points of $V_2$. For any connected graph $G$, we write $nG$ for the graph with $n$ components each isomorphic to $G$.
A graph $G$ is said to be \textit{Eulerian}, if it contains a closed trail which contains every edge of $G$ exactly once;  equivalently, $G$ is
Eulerian if and only if every vertex in $G$ has even degree.
A graph is said to be \textit{Hamiltonian}, if it contains a  cycle having all the vertices of the graph. A \emph{split graph} is a graph in which the vertices can be partitioned into a clique and an independent set.
For basic graph theory terminology, we refer to \cite{west}.

We recall the following theorem which will be use in the sequel.
\begin{thm}\normalfont(\cite[Proposition 7.2.3]{west})\label{l3}
If a graph $G$ is Hamiltonian, then $c(G-S)\leq |S|$ for every non-empty subset $S$ of $V(G)$.
\end{thm}

 For any  positive integer $n$,
$\tau(n)$ denotes the number of divisors of $n$ and $\sigma(n)$ denotes the sum of all divisors of $n$.
\noindent A positive integer $n$ is said to be \emph{deficient} if $\sigma(n)<2n$.

\begin{note}\label{n1} In some of the results of this paper, we will use the following basic facts which can be found in any basic number theory book; for instance, see \cite{mart}:
 Let $n=p_1^{\alpha_1}p_2^{\alpha_2}\ldots p_k^{\alpha_k}$ be a integer, where $p_i$'s are distinct primes and $\alpha_i\geq 1$. Then
$\displaystyle \tau(n)= \displaystyle\prod_{i=1}^k(\alpha_i+1)$. Moreover,
\begin{enumerate}[{\normalfont (i)}]
\item $\tau(n)$ is even if and only if $\alpha_i$ is odd for some $i\in\{1,2,\ldots,k\}$.
\item $\tau(n)$ is odd if and only if $\alpha_i$ is even for every $i\in\{1,2,\ldots,k\}$.
\item If $n$ is odd, then $\sigma(n)$ is odd (even) if and only if $\tau(n)$ is odd (even).
\item If $n=2^\alpha$, $\alpha\geq 1$, then $\sigma(n)$ is odd.
\item If $n=2^\alpha n'$, with $n'=p_2^{\alpha_2}\ldots p_k^{\alpha_k}$ is odd and $\alpha\geq 1$, then $\sigma(n)$ is odd (even) if and only if $\tau(n')$ is odd (even).
\end{enumerate}
\end{note}

\section{Some basic results}\label{sec:basic}
Recall that a group in which all the subgroups are normal is known as a Dedekind group. For the characterization of Dedekind groups the reader can refer to \cite[Theorem 5.3.7, p.143]{rob}.
For a given group $G$, we can define $\Gamma_N(G)$ only when $G$ is other than a Dedekind group; and we can define $\Gamma(G)$ only when
$G$ is other than the trivial group or a group of prime order. Also note that $|V(\Gamma (G)|= |L(G)|-2$, where $L(G)$ denotes the subgroup lattice of $G$.

Since any normal subgroup of a group permutes with every other subgroups, so we have the following result.
\begin{thm}\label{permutability nonabelian t1}
 Let $G$ be a finite non-simple group with $r$ proper normal subgroups. Then
\begin{equation*}
 \Gamma(G) \cong
\begin{cases}
		K_r+\Gamma_N(G), & \mbox{if~ } r\neq |L(G)|-2; \\
		K_r, &   \mbox{otherwise.}
	\end{cases}
\end{equation*}
\end{thm}

\begin{cor}\label{permutability nonabelian c1}
 Let $G$ be a finite non-simple group with $r$ proper normal subgroups. If $r\neq |L(G)|-2$, then
 the following holds:
\begin{enumerate}[{\normalfont (i)}]
\item $\deg_{\Gamma(G)}(H)$ $=
	\begin{cases}
		|L(G)|-3, & \mbox{if~} H \mbox{~is normal in~} G;\\
		r+\deg_{\Gamma_N(G)}(H), & \mbox{otherwise.}
	\end{cases}$.
\item $|E(\Gamma(G))|=|E(\Gamma_N(G))|+\displaystyle\frac{r}{2}~(2|L(G)|-r-5)$.
\item $\alpha(\Gamma(G))=\alpha(\Gamma_N(G))$.
\item $\omega(\Gamma(G))=r+\omega(\Gamma_N(G))$.
\item $\chi(\Gamma(G))=r+\chi(\Gamma_N(G))$.
\item $\gamma(G)=1$.
\end{enumerate}
\end{cor}
\begin{proof}
We prove part (ii) and the remaining parts of the result are immediate consequence of the previous theorem.  Note that the proper number
of non-normal subgroups of $G$ is $|L(G)|-r-2$. So by Theorem~\ref{permutability nonabelian t1},
\begin{equation*}\label{c111}
|E(\Gamma(G))|= {r \choose 2} +r[|L(G)|-r-2]+|E(\Gamma_N(G))|,
\end{equation*}
 which leads to the result by simplification.
\end{proof}

\section{Dihedral groups}\label{sec:dih}
The dihedral group of order $2n$ $ (n\geq 3)$ is defined by
\begin{equation*}
D_n=\langle a,b~|~a^n=b^2=1,ab=ba^{-1}\rangle.
\end{equation*}
 The subgroups of $D_n$ are listed below:
\begin{enumerate}[{\normalfont (i)}]
\item  cyclic groups $H_0^r:=\langle a^{\frac{n}{r}}\rangle$ of order $r$, where $r$ is a divisor of $n$;
\item  cyclic groups $H_i^1:=\langle ba^{i-1}\rangle$ of order $2$, where $i=1,2,\ldots,n$;
\item  dihedral groups $H_i^r:=\langle a^{\frac{n}{r}},ba^{i-1}\rangle$ of order $2r$, where $r$ is a divisor of $n$, $r\neq 1, n$ and $i=1,2,\ldots,\frac{n}{r}$.
\end{enumerate}
The proper normal subgroups of $D_n$ are the subgroups $H_0^r$,  $r\neq 1$ listed in (i), when $n$ is odd; the subgroups $H_0^r$, $r\neq 1$ listed in (i)
and $H_i^{\frac{n}{2}}$, $i=1,2$ of index 2, when $n$ is even. Thus
\begin{equation}\label{e60}
 |L(D_{n})|=\tau(n)+\sigma(n),
\end{equation}
\noindent and so
\begin{equation*}
|V(\Gamma(D_n)|=\tau(n)+\sigma(n)-2,
\end{equation*}
\begin{equation}\label{e110}
|V(\Gamma_N(D_n)|=
\begin{cases}
		\sigma(n)-1, & \mbox{if} ~n~  \mbox{is odd}; \\
		\sigma(n)-3, &   \mbox{otherwise.}
	\end{cases}
\end{equation}

The following details about the permutability of subgroups of $D_n$ were given by T$\check{a}$rn$\check{a}$uceanu in~\cite[p. 2513-2516]{Marius}.
We will use these to prove some of our results in this paper.
Consider the subgroups $H_i^r$ and $H_j^s$, where $r$ and $s$ are the divisor of $n$, $i \in \{1,2, \ldots , \frac{n}{r}\}$, $j \in \{ 1,2, \ldots, \frac{n}{s} \}$.
Then

 $H_i^rH_j^s=H_j^sH_i^r$ if and only if
$a^{2(i-j)}\in \langle a^{\frac{n}{[r,s]}}\rangle$, that is
\begin{equation}\label{e11}
\dfrac{n}{[r,s]}~\Big{|}~2(i-j).
\end{equation}
For a fixed divisor $r$ of $n$, and $i \in \{1,2, \ldots , \frac{n}{r}\}$, let $x_i^r$ denotes the number of solutions of~\eqref{e11}.
The value  of  $x_i^r$ is described explicitly in the following cases:

\noindent \textbf{Case 1.}  If $n$ is odd, then
 \begin{equation}\label{e17}
x_i^r= \displaystyle \sum_{s|n} \frac{[r,s]}{s}=r  \sum_{s|n} \frac{1}{(r,s)}
 \end{equation}

\noindent and so
\begin{equation}\label{e22}
 \sum_{r|n}\sum_{i=1}^{\frac{n}{r}}x_i^r=g(n),
\end{equation}
where $g$ is the function defined by $g(k)=k \displaystyle \sum_{r|k, s|k} \frac{1}{(r,s)}$, for all $k \in \mathds{N}$. Then $g$ is a multiplicative
function and
\begin{equation}\label{e181}
g(p^{\alpha})=\displaystyle  \frac {(2\alpha+1)p^{\alpha +2}-
(2\alpha +3)p^{\alpha +1}+ p +1}
{(p-1)^2}
\end{equation}

\noindent for any prime $p$  and $\alpha \in \mathds{N}$. If $n=p_1^{\alpha_1}p_2^{\alpha_2}\ldots p_k^{\alpha_k}$, where $p_i$'s are distinct
primes and $\alpha_i \geq 1$, then by~\eqref{e181}
\begin{equation}\label{e18}
g(n)=\displaystyle {\displaystyle \prod_{i=1}^k \frac {(2\alpha_i+1)p_i^{\alpha_i+2}-
(2\alpha_i+3)p_i^{\alpha_i+1}+p_i+1}
{(p_i-1)^2}}.
\end{equation}

\noindent \textbf{Case 2.} Let  $n$ be even.

\noindent \textbf{subcase 2a.} If $n=2^\alpha$, $\alpha\geq 2$, then
\begin{equation}\label{e50}
 x_i^r=2^{u+2}-2u+2\alpha-3
\end{equation}
for every $r=2^u$, where $0\leq u\leq \alpha-1$
\noindent and so
\begin{equation}\label{e51}
 \displaystyle \sum_{r|n} \sum_{i=1}^{\frac{n}{r}}x_i^r= (\alpha-1)2^{\alpha+3}+9.
\end{equation}

\noindent\textbf{subcase 2b.} If $n=2^\alpha n'$, with $n'$ is odd, $\alpha\geq 1$, then for any divisor $r$ of $n$ with $r=2^{\beta}r'$, where
$\beta \leq \alpha$ and $r'|n'$,
\begin{equation}\label{e13}
x_i^r =
	\begin{cases}
		(2^{\alpha+1}-1) x_i^{r'}, & \mbox{~~if } \beta=\alpha;
  \\
		(2^{\beta+2}-2\beta+2\alpha-3)x_i^{r'}, & \mbox{~~if } \beta<\alpha,
	\end{cases}
\end{equation} where $x_i^{r'}$ is given by~\eqref{e17} and so
\begin{equation}\label{e19}
\sum_{r|n}\sum_{i=1}^{\frac{n}{r}}x_i^r=[(\alpha-1)2^{\alpha+3}+9] g(n')
\end{equation}
In addition to these cases, for any integer $n \geq 3$, it is easy to see that

\begin{equation}\label{e12}
 x_1^n=\sigma(n)
\end{equation}
and for each $i \in \{1,2, \ldots , n\}$
\begin{equation}\label{e121}
 x_i^1=\tau(n).
\end{equation}
If $n$ is even, then for each $i=1,2$
\begin{equation}\label{e81}
x_i^{\frac{n}{2}}=\sigma(n)
\end{equation}

\subsection{Properties of $\Gamma_N(D_n)$}
In the following result, we describe the degrees of  the vertices of $\Gamma_N(G)$.
\begin{thm}\label{permutability nonabelian t40}
Let $n\geq 3$ be an integer.
\begin{enumerate}[{\normalfont (i)}]
\item  If $n$ is odd, then
$\deg_{\Gamma_N(D_n)}(H_i^r) =x_i^r-2$, for each divisor $r$ of $n$, $r\neq n$, $i=1,2,\ldots,\frac{n}{r}$, where $x_i^r$ is given by~\eqref{e17}.
\item  Let  $n$ be even. Then
 $\deg_{\Gamma_N(D_n)}(H_i^r)= x_i^r-4$, for each divisor $r$ of $n$, $r\neq n, \frac{n}{2}$, $i=1,2,\ldots,\frac{n}{r}$, where $x_i^r$ is given by~\eqref{e50}
if $n=2^\alpha$, $\alpha\geq 2$; $x_i^r$ is given by~\eqref{e13}
if $n=2^\alpha n'$ with $n'$ is odd, $\alpha\geq 1$.

\end{enumerate}
\end{thm}
\begin{proof}
One can observe that for each divisor $r$ of $n$ and $i\in \{1,2,\ldots,\frac{n}{r}\}$, the number of dihedral subgroups of $D_n$ permutes with $H_i^r$ is
$x_i^r$. So if $n$ is odd, then $\deg_{\Gamma_N(D_n)}(H_i^r)=x_i^r-|\{H_i^r, H_1^n\}|$ $=x_i^r-2$, where $x_i^r$ is given by \eqref{e17}.
If $n$ is even, then $\deg_{\Gamma_N(D_n)}(H_i^r)=x_i^r-|\{H_i^r, H_1^n, H_1^{\frac{n}{2}}, H_2^{\frac{n}{2}}\}|$ $=x_i^r-4$, where $x_i^r$ is given by \eqref{e50} and \eqref{e13}.
\end{proof}

\begin{cor}\label{permutability nonabelian c33}Let $n\geq 3$ be an integer and $g$ denotes the arithmetic function given by~\eqref{e18}.
\begin{enumerate}[{\normalfont (i)}]
\item If $n$ is odd, then
\begin{equation*}
|E(\Gamma_N(D_n))|=\displaystyle \frac {1}{2} \{g(n)-3\sigma(n)+2 \}.
\end{equation*}
\item  If $n=2^\alpha$, $\alpha\geq 2$, then
\begin{equation*}
|E(\Gamma_N(D_n))|=\displaystyle 2^\alpha(4\alpha-11)+14.
\end{equation*}
\item If $n=2^\alpha n'$, with $n'$ is odd, $\alpha\geq 1$, then
\begin{equation*}
|E(\Gamma_N(D_n))|=\displaystyle \frac{1}{2}\big \{[(\alpha-1)2^{\alpha+3}+9] g(n')-7\sigma(n)+12\big\}.
\end{equation*}

\end{enumerate}
\end{cor}
\begin{proof}
 {(i)}: If $n$ is odd, then by Theorem~\ref{permutability nonabelian t40}(i) and by using \eqref{e22}, \eqref{e12}, we have
\begin{align*}
|E(\Gamma_N(D_n))| &=\displaystyle \frac{1}{2}\sum_{H\in \Gamma_N(D_n)}\deg_{\Gamma_N(D_n)}(H)\\
&=\frac{1}{2}\sum_{\substack{r|n\\  r\neq n}}\sum_{i=1}^{\frac{n}{r}}(x_i^r-2)\\
&=\frac{1}{2}\big [\sum_{r|n}\sum_{i=1}^{\frac{n}{r}}(x_i^r-2)-(x_1^n-2)\big]\\
&=\frac{1}{2}\big [\sum_{r|n}\sum_{i=1}^{\frac{n}{r}}x_i^r-2\sum_{r|n}\sum_{i=1}^{\frac{n}{r}}1-(\sigma(n)-2)\big]\\
&=\frac{1}{2}\big [\sum_{r|n}\sum_{i=1}^{\frac{n}{r}}x_i^r-2\sigma(n)-(\sigma(n)-2)\big]\\
&=\frac{1}{2}\big [g(n)-3\sigma(n)+2]\big.
\end{align*}
\noindent{(ii)-(iii)}: By Theorem~\ref{permutability nonabelian t40}(ii) and by using \eqref{e12} and \eqref{e81}, we have
 \begin{align}\label{e52}
|E(\Gamma_N(D_n))|&=\displaystyle \frac{1}{2}\sum_{H\in \Gamma_N(D_n)}\deg_{\Gamma_N(D_n)}(H)\nonumber \\
&=\frac{1}{2}\sum_{\substack{r|n\\ r\neq n,\frac{n}{2}}}\sum_{i=1}^{\frac{n}{r}} (x_i^r-4)\nonumber \\
&=\frac{1}{2}\big[\sum_{r|n}\sum_{i=1}^{\frac{n}{r}}(x_i^r-4)-(x_1^n-4)-\sum_{r=\frac{n}{2}}\sum_{i=1}^{\frac{n}{r}}(x_i^r-4)\big]\nonumber \\
&=\frac{1}{2}\big[\sum_{r|n}\sum_{i=1}^{\frac{n}{r}}x_i^r-4\sum_{r|n}\sum_{i=1}^{\frac{n}{r}}1-(\sigma(n)-4)-2(\sigma(n)-4)\big]\nonumber \\
&=\frac{1}{2}\big[\sum_{r|n}\sum_{i=1}^{\frac{n}{r}}x_i^r-4\sigma(n)-(\sigma(n)-4)-2(\sigma(n)-4)\big]\nonumber\\
&=\frac{1}{2}\big [\sum_{r|n}\sum_{i=1}^{\frac{n}{r}}x_i^r-7\sigma(n)+12\big].
\end{align}
Now to complete the proof it remains to consider the following cases:

\noindent \textbf{Case a.} If $n=2^\alpha$, $\alpha\geq 2$, then \eqref{e52} reduces to the result after simplification, by using \eqref{e51} and  $\sigma(n)=2^{\alpha+1}-1$.

\noindent \textbf{Case b.} If $n=2^\alpha n'$, with $n'$ is odd and $\alpha\geq 1$, then \eqref{e52} reduces to the result after simplification, by using \eqref{e19}.
\end{proof}
In the following result, we determine
the values of $n$ for which $\Gamma_N(D_n)$ is Eulerian.
\begin{cor}\label{permutability nonabelian c39}
Let $n\geq 3$ be an integer.
\begin{enumerate}[{\normalfont (i)}]
\item If $n$ is odd with $n=p_1^{\alpha_1}p_2^{\alpha_2}\ldots p_k^{\alpha_k}$, where $p_i$'s are distinct primes and $\alpha_i\geq 1$,
then $\Gamma_N(D_n)$ is Eulerian if and only if $\alpha_i$ is odd for some $i \in \{1,2, \ldots , k\}$.
 \item If $n=2^{\alpha}$, $\alpha\geq 2$, then $\Gamma_N(D_n)$ is non-Eulerian.
\item If $n=2^\alpha n'$, with $n'=p_1^{\alpha_1}p_2^{\alpha_2}\ldots p_k^{\alpha_k}$ is odd, where $p_i$'s are distinct primes and $\alpha\geq 1$, then $\Gamma_N(D_n)$ is Eulerian if and only if $\alpha_i$ is
odd for some $i \in \{1,2, \ldots , k\}$.
\end{enumerate}
\end{cor}
\begin{proof}
 (i): By Theorem~\ref{permutability nonabelian t40}(i),
$\deg_{\Gamma_N(D_n)}(H_i^r)$  is even if and only if $\displaystyle \sum_{s|n}\frac{[r,s]}{s}$  is even. Since $n$ is odd, $\displaystyle \frac{[r,s]}{s}$
is odd for each divisor $r$ and $s$ of $n$. So it follows that $\displaystyle \sum_{s|n}\frac{[r,s]}{s}$ is even only when  $\tau(n)$ is even;
this is true only when $\alpha_i$ is odd, for some $i \in \{1,2, \ldots , k\}$.

\noindent{(ii)}: By Theorem~\ref{permutability nonabelian t40}(ii), $\deg_{\Gamma_N(D_n)}(H_i^r)$ is odd, for every $\alpha$ and so $\Gamma_N(D_n)$ is non-Eulerian.

\noindent{(iii)}: By Theorem~\ref{permutability nonabelian t40}(ii), $\deg_{\Gamma_N(D_n)}(H_i^r)$ is even if and only if
$x_i^r$ is even; that is when $x_i^{r'}$ is even, where
$x_i^{r'}$ is given by~\eqref{e17}. But for each divisor $r'$ of $n'$, $\frac{[r',s']}{s'}$ is always odd. So $x_i^{r'}$is even  if and only if
$\tau(n')$ is even; that is only when $\alpha_i$ is odd, for some $i$.
\end{proof}

Now we make further investigation to know more about the structure of $\Gamma_N(D_n)$. If $n$ is odd, then for each divisor $r$ of $n$, $r\neq n$, let $\mathcal A_r^o=\{H_i^r~|~i=1,2,\ldots,\frac{n}{r}\}$. It is easy to see that $\mathcal A_r^o$ is
a maximal independent set in $\Gamma_N(D_n)$. Also these sets are mutually disjoint and
 \begin{equation*}
 \Gamma_N(D_n)= \displaystyle \bigcup_{\substack{r|n\\  r\neq n}} \mathcal A_r^o.
 \end{equation*}
Also note that the number of such $A_r^o$ is $\tau(n)-1$.

\noindent Let $n=2^\alpha n'$ be even, with $n'$ is odd and
$\alpha \geq 1$. Then for every divisor $r$ of $n$ with $\frac{n}{r}$ is even and $r\neq \frac{n}{2}$, let
$\mathcal A_r^e:=\{H_i^r~|~i=1,2,\ldots, \frac{n}{2r}\}$ and $\mathcal B_r^e:=\{H_i^r~|~i=\frac{n}{2r}+1, \frac{n}{2r}+2,
\ldots, \frac{n}{r}\}$. For every divisor $r$ of $n$ with $\frac{n}{r}$ is odd and $r\neq n$, let $\mathcal C_r^e:=\{H_i^r~|~i=1,2,\ldots, \frac{n}{r}\}$.
 Here each of $\mathcal A_r^e$, $\mathcal B_r^e$
and $\mathcal C_r^e$ forms a maximal independent set in $\Gamma_N(D_n)$. Also these three class of sets are mutually disjoint and
\begin{equation*}
 \Gamma_N(D_n)= \displaystyle \bigcup_{\substack{r|n\\ r \neq n\\ \frac{n}{r} ~is~odd }} \mathcal C_r^e \cup \bigcup_{\substack{r|n\\  r \neq \frac{n}{2}\\ \frac{n}{r}~is~even}} (\mathcal A_r^e\cup \mathcal B_r^e).
  \end{equation*}

The number of divisors $r$ of $n$ such that $r\neq \frac{n}{2}$ with $\frac{n}{r}$ is even is $\alpha\tau(n')-1$. Each of these divisors gives rise
to two partition namely, $A_r^e$ and $B_r^e$ in $V(\Gamma_N(D_n))$. The number of  divisors $r$ of $n$ such that $r\neq n$ with $\frac{n}{r}$ is odd is
$\tau(n')-1$. Each of these divisors gives rise
to exactly one partition namely, $C_r^e$ in $V(\Gamma_N(D_n))$. Thus in total, we have $2(\alpha\tau(n')-1)+\tau(n')-1=(2\alpha+1)\tau(n')-3$ partitions of
$V(\Gamma_N(D_n))$.
\begin{thm}\label{t33}
Let $n\geq 3$ be an integer. Then
$\Gamma_N(D_n)$ is  $(\tau(n)-1)$-partite if $n$ is odd; $((2\alpha+1)\tau(n')-3)$-partite if $n=2^\alpha n'$, with $n'$ is odd, $\alpha \geq 1$. But $\Gamma_N(D_n)$ is not a complete partite graph for any $n$.
\end{thm}
\begin{proof}
Partiteness of $\Gamma_N(D_n)$ follows from the above discussions. Now we prove the last part of the theorem.
Let $n=p_1^{\alpha_1}p_2^{\alpha_2}\ldots p_k^{\alpha_k}$, where $p_i$'s are distinct primes and $\alpha_i\geq 1$.

\noindent\textbf{Case 1.} $n$ is odd. If $k=1$ and $\alpha_1=1$, then by Theorem~\ref{t35}, $\Gamma_N(D_n)$ is not complete partite graph. Otherwise,
let $H_1^1\in \mathcal A_1$ and $H_2^r\in \mathcal A_r$, where $r$ is a divisor of $n$, $r\neq n$. Here $H_1^1$ does not permutes with $H_2^r$ and so
$\Gamma_N(D_n)$ is not a complete partite graph.

\noindent\textbf{Case 2.} $n$ is even. Let $H_1^1\in \mathcal A_1$ and $H_{2+\frac{n}{2}}^1\in \mathcal A_{1+\frac{n}{2}}$. Here $H_1^1$ does not
permute with $H_{2+\frac{n}{2}}^1$ and so $\Gamma_N(D_n)$ is not a complete partite graph.
\end{proof}
\begin{thm}\label{t35}
Let $n\geq 3$ be an integer. Then
 $\Gamma_N(D_n)$ is totally disconnected if and only if $n$ is an odd prime.
\end{thm}
\begin{proof}
 Let $n=p_1^{\alpha_1}p_2^{\alpha_2}\ldots p_k^{\alpha_k}$, where $p_i$'s are distinct primes and $\alpha_i\geq 1$.

\noindent\textbf{Case 1.} $k=1$.

\noindent\textbf{subcase 1.} $n$ is odd. If $\alpha_1=1$, then no $H_i^1$ permutes with any $H_j^1$, for every $i, j\in \{1,2,\ldots,p\}$, $i\neq j$ and so
$\Gamma_N(D_p)\cong \overline{K}_p$. If $\alpha>1$, then $H_1^p$, $H_1^1$ permutes. So $\Gamma_N(D_n)$ contains an edge.

\noindent\textbf{subcase 2.} $n$ is even. Then $H_1^1$ and $H_{1+\frac{n}{2}}^1$ permutes. Therefore $\Gamma_N(D_n)$ contains an edge.

\noindent\textbf{Case 2.} $k\geq 2$. If $n$ is odd, then $H_1^{p^{\alpha_1}}$ and $H_1^{p^{\alpha_2}}$ permutes and so $\Gamma_N(D_n)$ contains an edge.
If $n$ is even, then $H_1^1$, $H_{1+\frac{n}{2}}^1$ permutes and so $\Gamma_N(D_n)$ contains an edge. Combaining all the cases together gives the result.
\end{proof}
%
%

\begin{thm}\label{t37}
Let $n\geq 3$ be an integer.

\begin{enumerate}[{\normalfont (i)}]
\item $\alpha(\Gamma_N(D_n))=
 	\begin{cases}
 		n, & \mbox{if~} n \mbox{~is odd};  \\
 		\frac{n}{2}, & \mbox{if~} n \mbox{~is even.}
 	\end{cases}$
\item $\gamma(\Gamma_N(D_n))= p$, where $p$ is the smallest prime factor of $n$.
\item $\omega(\Gamma_N(D_n))=\chi(\Gamma_N(D_n))=
 	\begin{cases}
 		\tau(n)-1,&\mbox{if} ~n~ \mbox{is odd};  \\
 		(2\alpha+1)\tau(n')-3, & \mbox{if~} n=2^\alpha n', \mbox{with~} n' \mbox{~is odd.}
  	\end{cases}$

In particular, $\Gamma_N(D_n)$ is weakly perfect.
\end{enumerate}
\end{thm}
\begin{proof}
\noindent{(i):} 
As explained earlier, $V(\Gamma_N(D_n))$ is the disjoint union of maximal independent sets. If $n$ is odd, the maximal independent set $\mathcal A_1^o$ in $\Gamma_N(D_n)$ has the maximal cardinality with $|\mathcal A_1^o|=n$. If $n$ is even, the maximal independent sets $\mathcal A_1^e$ and $\mathcal B_1^e$ in $\Gamma_N(D_n)$ both attains the maximal cardinality with $|\mathcal A_1^e|= \frac{n}{2}=|\mathcal B_1^e|$.

\noindent{(ii):} Note that $\omega(\Gamma_N(D_n))\leq \chi(\Gamma_N(D_n))$, for every $n$.

\noindent\textbf{Case 1.} $n$ is odd. Then $\mathcal A:=\{H_1^r~|~r|n$, $r\neq n\}$ is a clique set in $\Gamma_N(D_n)$. So
$\omega(\Gamma_N(D_n))\geq |\mathcal A|=\tau(n)-1$. By Theorem~\ref{t33}, $\Gamma_N(D_n)$ is a $\tau(n)-1$ partite graph. So $\chi(\Gamma_N(D_n))\leq \tau(n)-1$.
Thus $\tau(n)-1\leq \omega(\Gamma_N(D_n))\leq \chi(\Gamma_N(D_n))\leq \tau(n)-1$. Therefore, $\omega(\Gamma_N(D_n))=\chi(\Gamma_N(D_n))=\tau(n)-1$.

\noindent\textbf{Case 2.} $n=2^\alpha p_2^{\alpha_2}\ldots p_k^{\alpha_k}$, where $p_i$'s are distinct primes and $\alpha$, $\alpha_i\geq 1$.


 Then Let $\mathcal A:=\{H_1^r, H_{1+\frac{n}{2r}}^r~|~r$ is a divisor
of $n$ with $\frac{n}{r}$ is even, $r\neq n$, $2^{\alpha-1}p_2^{\alpha_2}\ldots p_k^{\alpha_k}\}\cup \{H_1^r~|~ r$ is a divisor of $n$ with $\frac{n}{r}$
is odd and $r\neq n\}$ forms a clique set in $\Gamma_N(D_n)$.

In this case, $\omega(\Gamma_N(D_n))\geq |\mathcal A|=(2\alpha+1)\tau(n')-3$. By Theorem~\ref{t33}, $\Gamma_N(D_n)$ is a
$(2\alpha+1)\tau(n')-3$ partite graph and so $\chi(\Gamma_N(D_n))\leq (2\alpha+1)\tau(n')-3$.
Thus $(2\alpha+1)\tau(n')-3\leq \omega(\Gamma_N(D_n))\leq \chi(\Gamma_N(D_n))\leq (2\alpha+1)\tau(n')-3$.
Therefore, $\omega(\Gamma_N(D_n))=\chi(\Gamma_N(D_n))=(2\alpha+1)\tau(n')-3$.

Weakly perfectness of $\Gamma_N(D_n)$ follows by the above two cases.

\noindent{(iii):} By Theorem~\ref{t33}, $\Gamma_N(D_n)$ is a partite graph and every partition of $V(\Gamma_N(D_n))$ is an maximal independent set. Among these maximal independent sets, $\mathcal A_d$ is a maximal independent set in $\Gamma_N(D_n)$ with minimum cardinality, where $d$ is the largest divisor of $n$.
It is well known that in a graph any maximal independent set is a dominating set.
Consequently, $\mathcal A_d$ is an independent dominating set of $\Gamma_N(D_n)$ with minimum cardinality, where $d$ is the largest divisor of $n$.

If  $n=2^{\alpha}p_1^{\alpha_1}p_2^{\alpha_2}\ldots p_k^{\alpha_k}$ is even, where $p_i$'s are distinct primes and $\alpha$, $\alpha_i\geq 1$, then $d=2^{\alpha-1}p_2^{\alpha_2}\ldots,p_k^{\alpha_k}$ is the largest divisor of $n$ and so $|\mathcal A_d|=2$. Hence
$\gamma(\Gamma_N(D_n))\leq 2$. Suppose $\gamma(\Gamma_N(D_n))=1$, then there exists a subgroup in $D_n$, which permutes with every other subgroups of $D_n$, which is not possible, since by Theorem~\ref{t33}, $\Gamma_N(D_n)$ is a partite graph with every partition has at least two vertices. So $\gamma(\Gamma_N(D_n))=2$.

 Let $n=p_1^{\alpha_1}p_2^{\alpha_2}\ldots p_k^{\alpha_k}$ be odd, where $p_i$'s are distinct primes and $\alpha_i\geq 1$. With out loss of generality we assume that $p_1<p_2 < \ldots < p_k$.  Then $d=p_1^{\alpha_1-1}p_2^{\alpha_2}\ldots p_k^{\alpha_k}$ is the largest divisor of $n$ and so $|\mathcal A_d|=p_1$. Thus $\gamma(\Gamma_N(D_n))\leq p_1$. We show that, $\mathcal A_d$ is a dominating set with minimal cardinality. Now for any $i\in \{1,2,\ldots,\frac{n}{d}\}$, we consider $\mathcal A_d-\{H_i^d\}$. Then no vertex in $\mathcal A_d-\{H_i^d\}$ permutes with $H_i^1$. It follows that $p_1-1<\gamma(\Gamma_N(D_n))\leq p_1$. Therefore, $\gamma(\Gamma_N(D_n))=p_1$.
\end{proof}

In the next result, we give the structure of $\Gamma_N(D_n)$ for some values of $n$.

\begin{thm}\label{permutability nonabelian t36}
Let $n\geq 3$ be an integer.

\begin{enumerate}[{\normalfont (i)}]
\item If $n=2^\alpha$, $\alpha\geq2$, then
\begin{align}\label{e71}
\Gamma_N(D_n)\cong 2 (\Gamma_N(D_{\frac{n}{2}})+K_2)
\end{align}
 with $\Gamma_N(D_{2^2})\cong 2K_2$.
\item If $n=p^{\alpha}$, where $p$ is a prime, $p>2$, $\alpha\geq1$, then
\begin{align}\label{e72}
\Gamma_N(D_n)\cong p (\Gamma_N(D_{\frac {n}{p}})+K_1)
\end{align}
 with $\Gamma_N(D_p)\cong \overline{K}_p$.
\item If $n=2p$, where $p$ is a prime, $p>2$, then
\begin{align}\label{e73}
\Gamma_N(D_n)\cong pK_3.
\end{align}
\item If $n=pq$, where $p, q$ are distinct primes, $2<p<q$, then
\begin{align}\label{e74}
 \Gamma_N(D_n)\cong \displaystyle\bigcup_{i=1}^{p}\bigcup_{j=1}^{q} \mathcal{G}_{ij},
\end{align}
 where for each $i=1,2,\ldots,p$, $j=1,2,\ldots,q$, $\mathcal{G}_{ij}$ is the complete
graph with vertex set $\{u_i,v_j,w_{ij}\}$.
\end{enumerate}
\end{thm}
\begin{proof}

\noindent(i): If $\alpha=2$, then $H_i^1$, $i=1,2,3,4$ are the only non-normal subgroups of $D_n$.
Also $H_i^1$ and $H_{i+2}^1$
 permutes with each other for each $i\in \{1,2\}$; no two remaining subgroups permutes. It follows that, $\Gamma_N(D_n)\cong 2K_2$.
 Now we consider $\alpha>2$.
For each $i=1,2$, let $\mathcal G_i$ denotes the subgraph of $\Gamma_N(D_n)$ induced by the proper non-normal subgroups in
$H_i^{2^{\alpha-1}}$. Clearly $\Gamma_N(D_n)\cong \mathcal G_1\cup \mathcal G_2$. Since for each $i=1,2$ $H_i^{2^{\alpha-1}}\cong D_{2^{\alpha-1}}$, so we have
\begin{equation}\label{e100}
\Gamma_N(H_i^{2^{\alpha-1}})\cong \Gamma_N(D_{2^{\alpha-1}}).
\end{equation} For each $i=1,2$, $H_i^{2^{\alpha-2}}$ and $H_{i+2}^{2^{\alpha-2}}$ permutes in $D_n$ and they are normal in $H_i^{2^{\alpha-1}}$;
but are not normal in $D_n$. So by using~\eqref{e100}, for each $i=1,2$ we have
\begin{equation*}
\mathcal G_i\cong \Gamma_N(H_i^{2^{\alpha-1}})+K_2\cong \Gamma_N(D_{2^{\alpha-1}})+K_2.
\end{equation*}
Here no vertex of $\mathcal G_1$ is adjacent with any vertex of $\mathcal G_2$. Thus $\Gamma_N(D_n)$ is the disjoint union of $\mathcal G_1$ and $\mathcal G_2$.

 \noindent (ii): If $\alpha=1$, then $H_i^1$, $i=1,2,\dots,p$ are the only non-normal
subgroups of $D_n$. No two distinct subgroups of
this type permutes. Therefore, $\Gamma_N(D_n)\cong \overline{K}_p$.
Now we consider $\alpha>1$. For each $i=1,2,\ldots,p$, let $\mathcal G_i$ denotes the subgraph of $\Gamma_N(D_n)$ induced by the proper
non-normal subgroups in $H_i^{p^{\alpha-1}}$.
Clearly $\Gamma_N(D_n)\cong \displaystyle \bigcup_{i=1}^{p} \mathcal G_i$. Since for each $i=1,2,\dots,p$, $H_i^{p^{\alpha-1}}\cong D_{p^{\alpha-1}}$,
so we have
\begin{equation}\label{e101}
 \Gamma_N(H_i^{p^{\alpha-1}})\cong \Gamma_N(D_{p^{\alpha-1}}).
\end{equation}
 Here
$H_i^{p^{\alpha-1}}$, $i=1,2,\ldots,p$ permutes with all its proper subgroups. Further no $H_i^{\alpha-1}$ permutes with any $H_j^{\alpha-1}$ and
its subgroups for all $i$, $j=1,2,\ldots,p$, $i\neq j$. So by using~\eqref{e101}, for each $i=1,2,\ldots,p$, we have
\begin{equation*}
 \mathcal G_i\cong \Gamma_N(H_i^{p^{\alpha-1}})+K_1
\cong \Gamma_N(D_{p^{\alpha-1}})+K_1.
\end{equation*}
So $\Gamma_N(D_n)$ is the disjoint union of $\mathcal G_i$'s.\\
\noindent{(iii)}: If $n=2p$, then $H_i^2$, $i=1,\ldots,p$ and $H_i^1$, $i=1,2,\ldots,2p$ are the only non-normal
subgroups of $D_n$. Also for each $i\in \{1,\ldots,p\}$, $H_i^1$, $H_{i+p}^1$ are subgroups of $H_i^2$;
they permutes with each other; no $H_i^2$ permutes with $H_j^2$ for every $i\neq j$. It follows that
$\Gamma_N(D_n)\cong pK_3$.

\noindent{(iv)}: If $n=pq$, $2 < p < q$, then $H_i^q$, $i=1,\ldots,p$, $H_j^p$, $j=1,\ldots,q$ and $H_k^1$,
 $k=1,\ldots,pq$ are the only non-normal subgroups of $D_n$. Here for each $i\in \{1,\ldots,p\}$, $j\in \{1,\ldots,q\}$, $H_i^q$
permutes with $H_j^p$; no two subgroups of the form $H_i^q$ permutes with each other. For each
$i\in \{1,\ldots,p\}$, $j\in \{1,\ldots,q\}$, by Chinese Remainder Theorem, there exist a positive integer, say $m_{ij}$ such that $m_{ij}\equiv
i \mod{p}$ and $m_{ij}\equiv j \mod{q}$. It follows that, $H_{m_{ij}}^1$ is the unique proper subgroup of $H_i^q$ and
 $H_j^p$ for every $i=1,2,\ldots,p$, $j=1,2,\ldots,q$. So for each $i\in \{1,\ldots,p\}$, $j\in \{1,\dots,q\}$, $H_i^q$,
$H_j^p$, $H_{m_{ij}}^1$
 permutes with each other and hence they form the complete graph, say $\mathcal G_{ij}$ as a subgraph of $\Gamma_N(D_n)$ with
vertex set $\{H_i^q, H_j^p, H_{m_{ij}}^1\}$. Thus $\Gamma_N(D_n)$ is the union of these $\mathcal G_{ij}$'s.
\end{proof}

\begin{example}\label{e222}
In Figure 1, we exhibit the structure of $\Gamma_N(D_{15})$ described in Theorem~\ref{permutability nonabelian t36}\normalfont{(iv)}.
   \begin{figure}[ ht ]
   \begin{center}
    \includegraphics[scale=0.8]{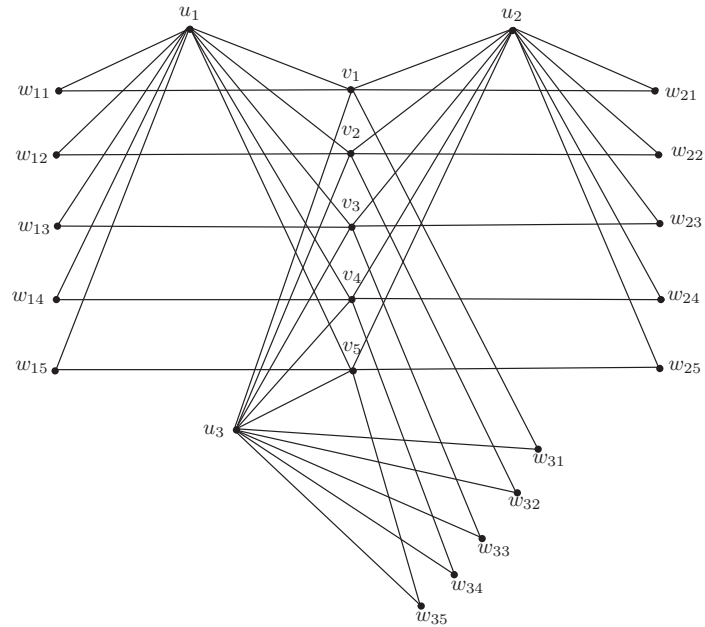}
   \caption{ The graph $\Gamma_N(D_{15})$}
   \end{center}
   \end{figure}
 \end{example}

\begin{thm}\label{l7}
 If $n\geq 3$ is a deficient number, then $\Gamma_N(D_n)$ is non-Hamiltonian.
\end{thm}
\begin{proof}
 Let $S=V(\Gamma_N(D_n))-\mathcal A_1$. Since $n$ is an deficient number, we have
$|S|=\sigma(n)-(n+1)<n=|\mathcal A_1|=c(\Gamma_N(D_n)-S)$ and so by Theorem~\ref{l3},
$\Gamma_N(D_n)$ is non-Hamiltonian.
\end{proof}

\begin{cor}\label{permutability nonabelian c14}
 Let $p$, $q$ be two distinct primes. If $n$ is one of the following: $n=2^{\alpha}$, $\alpha\geq2$,
 $n=p^{\alpha}$, $\alpha\geq 1$, $n=2p$, $p>2$, or $n=pq$, $2< p< q$, then $\Gamma_N(D_n)$ is non-Hamiltonian.
\end{cor}
\begin{proof}
 If $n$ is one of the following:
 $n=2^{\alpha}$, $\alpha\geq2$, $n=p^\alpha$, $\alpha\geq 1$, or $n=2p$, $p>2$, then by~\eqref{e71}, \eqref{e72}, and \eqref{e73}, respectively
$\Gamma_N(D_n)$ is disconnected and hence $\Gamma(D_n)$ is non-Hamiltonian.

If $n=pq$, then $n$ is a deficient number and so by Theorem~\ref{l7}, $\Gamma_N(D_n)$ is non-Hamiltonian.
\end{proof}

\begin{thm}\label{c1}
 If $n=2^\alpha$,  $\alpha\geq 2$, then each component of $\Gamma_N(D_n)$ contains a Hamiltonian path.
\end{thm}
\begin{proof}
We shall prove this result by induction on $\alpha$. If $\alpha=2$, then by Theorem~\ref{permutability nonabelian t36}(i), $\Gamma_N(D_n) \cong 2K_2$.
Here each component contains a Hamiltonian path. Hence the result is true for $\alpha=2$. We assume that the result is true for any positive integer $m<\alpha$.
Now we consider $\alpha>2$. By Therorem~\ref{permutability nonabelian t36}(i), $\Gamma_N(D_n)$  has two components, say $G_1$ and $G_2$ each isomorphic to $ \Gamma_N(D_{2^{\alpha-1}})+K_2$.
Let us consider any one of these components, say $G_1$. By induction hypothesis and by Therorem~\ref{permutability nonabelian t36}(i), $\Gamma_N(D_{2^{\alpha-1}})$ has two components and each contains a Hamiltonian path, say $P$ and $P'$ respectively. Also let $u$ and $v$ be the vertices of $K_2$.
Then we have a Hamiltonian path, $u-P-v-P'$ in $G_1$. This completes the proof.
\end{proof}

\subsection{Properties of $\Gamma(D_n)$}
Now we start to investigate the structure and properties of $\Gamma(D_n)$.
\begin{thm}\label{t11}
Let $n\geq 3$ be an integer. Then
$\Gamma(D_n)\cong K_r+\Gamma_N(D_n)$,
where $r=\tau(n)+1$ if $n$ is even; $r=\tau(n)-1$ if $n$ is odd.
\end{thm}
\begin{proof}
Since the number of normal subgroups of $D_n$ is $\tau(n)+1$ if $n$ is even; $\tau(n)-1$ if $n$ is odd, so the proof follows by
Theorems~\ref{permutability nonabelian t1} and \ref{t33}.
\end{proof}

\begin{cor}\label{permutability nonabelian c24}
Let $n\geq 3$ be an integer.
\begin{enumerate}[{\normalfont (i)}]
\item $\deg_{\Gamma(D_n)}(H_o^r)=\tau(n)+\sigma(n)-3$, for every divisor $r$ of $n$, $r\neq 1$.
\item  $\deg_{\Gamma(D_n)}(H_i^r)=\tau(n)+x_i^r-3$, for every divisor $r$ of $n$, $r\neq n$, $ i=1,2,\ldots,\frac{n}{r}$, where $x_i^r$ is given by~\eqref{e17} if $n$ is odd;  $x_i^r$
is given by \eqref{e50} if $n=2^{\alpha}$, $\alpha\geq 2$; $x_i^r$ is given by \eqref{e13} if $n=2^\alpha n'$, with $n'$ is odd and $\alpha\geq 1$.
\end{enumerate}
\end{cor}
\begin{proof}
{(i)}: For each divisor $r$ of $n$, $r\neq 1$, the subgroup $H_0^r$ is normal in $D_n$ for any $n$. So by Corollary~\ref{permutability nonabelian c1}(i) and \eqref{e60}, $\deg_{\Gamma(D_n)}(H_0^r)=\tau(n)+\sigma(n)-3$.

\noindent(ii): Follows by Theorem~\ref{permutability nonabelian t40}, Corollary~\ref{permutability nonabelian c1}(i) and \eqref{e60}.
\end{proof}

\begin{cor}\label{permutability nonabelian c34}Let $n\geq 3$ be an integer and $g$ denotes the arithmetic function given by~\eqref{e18}.
\begin{enumerate}[{\normalfont (i)}]
\item If $n$ is odd, then
\begin{equation*}
|E(\Gamma(D_n))|=\displaystyle \frac{1}{2}\big \{g(n)-5\sigma(n)+\tau(n)[\tau(n)+2\sigma(n)-5]+6 \big \}.
\end{equation*}
\item If $n=2^\alpha$, $\alpha\geq 2$, then
\begin{equation*}
|E(\Gamma(D_n))|=\displaystyle \frac{1}{2}\big \{2^{\alpha+1}(6\alpha-7)+\alpha^2-5\alpha+14\big \}.
\end{equation*}
\item  If $n=2^\alpha n'$, with $n'$ is odd, $\alpha \geq 1$, then
\begin{equation*}
\begin{split}
|E(\Gamma(D_n))|=\displaystyle \frac{1}{2}\big \{[(\alpha-1)2^{\alpha+3}+9] g(n')-5\sigma(n)+\tau(n)+6
+\tau(n)[\tau(n)+2\sigma(n)-6]\big \}.
\end{split}
\end{equation*}
\end{enumerate}
\end{cor}
\begin{proof}
Follows by Corollaries~\ref{permutability nonabelian c33}, \ref{permutability nonabelian c1}(ii), \eqref{e60} and from the fact that the number of normal subgroups of $D_n$ is $\tau(n)+1$ if $n$ is even; $\tau(n)-1$ if $n$ is odd.
\end{proof}

\begin{cor}\label{permutability nonabelian c22}
Let $n\geq 3$ be an integer.
\begin{enumerate}[{\normalfont (i)}]
 \item If $n$ is odd, then  $\Gamma(D_n)$ is non-Eulerian.

\item  If $n=2^\alpha$, $\alpha\geq 2$, then $\Gamma(D_n)$ is Eulerian if and only if $\alpha$ is odd.

\item  If $n=2^\alpha n'$, with $n'=p_1^{\alpha_1}p_2^{\alpha_2}\ldots p_k^{\alpha_k}$  is odd, where $p_i$'s are distinct primes and $\alpha, \alpha_i\geq 1$,
then $\Gamma(D_n)$ is Eulerian if and only if $\alpha$ is odd and $\alpha_i$ is even for every $i \in \{ 1,2, \ldots , k\}$.
\end{enumerate}
\end{cor}
\begin{proof}
(i): By Corollary~\ref{permutability nonabelian c24}(ii) and \eqref{e121}, $\deg_{\Gamma(D_n)}(H_i^1)$ is odd, for each $i=1,2,\ldots,n$. So
$\Gamma(D_n)$ is non-Eulerian.

\noindent{(ii)}:  By Corollary~\ref{permutability nonabelian c24}(i), $\deg_{\Gamma(D_n)}{(H_0^r)}$ is even if and only if either $\tau(n)$ is odd and $\sigma(n)$ is even or
$\tau(n)$ is even and $\sigma(n)$ is odd. By Note~\ref{n1}, it follows that $\deg_{\Gamma(D_n)}(H_0^r)$ is even if and only if $\alpha$ is odd.
 Also by Corollary~\ref{permutability nonabelian c24}(ii),
$\deg(H_i^r)$ is even if and only if either $\tau(n)$ is odd and $x_i^r$ is even or $\tau(n)$ is even and $x_i^r$ is odd. By \eqref{e50}
$x_i^r$ is odd. Further $\tau(n)$ is even only when $\alpha$ is odd. Thus $\deg(H_i^r)$ is even only when $\alpha$ is odd.
Combining all these, we get the result.

\noindent{(iii)}: By Corollary~\ref{permutability nonabelian c24}(i), $\deg_{\Gamma(D_N)}{(H_0^r)}$ is even if and only if either $\tau(n)$ is odd and $\sigma(n)$ is even or
$\tau(n)$ is even and $\sigma(n)$ is odd. By Note~\ref{n1}, we have $\tau(n)$ is even and $\sigma(n)$ is odd. It follows that $\deg_{\Gamma(D_n)}(H_0^r)$
is even if and only if $\alpha$ is odd and $\alpha_i$ is even for every $i\in\{1,2,\ldots,k\}$. Also by
Corollary~\ref{permutability nonabelian c24}(ii), $\deg_{\Gamma(D_n)}{(H_i^r)}$ is even if and only if either $\tau(n)$ is odd and $x_i^r$ is even or
$\tau(n)$ is even and $x_i^r$ is odd. By the above argument, we have $\tau(n)$ is even and so we must have $x_i^r$ is odd. By \eqref{e13}, $x_i^r$ is odd if and only if $\sigma(n')$ is odd. Now, by Note~\ref{n1}, it follows that $\deg_{\Gamma(D_n)}(H_i^r)$ is even if and only if $\alpha_i$ is even for every $i\in\{1,2,\ldots,k\}.$
Proof follows by combining the above facts.
\end{proof}

\begin{cor}\label{permutability nonabelian c5}
Let $n\geq 3$ be an integer.
\begin{enumerate}[{\normalfont (i)}]
\item $\alpha(\Gamma(D_n))=
 	\begin{cases}
 		n, & \mbox{if~ } n \mbox{~is odd};  \\
 		\frac{n}{2}, & \mbox{if~ } n \mbox{~is even}
 	\end{cases}$
\item $\omega(\Gamma(D_n))=\chi(\Gamma(D_n))=
 	\begin{cases}
 		2(\tau(n)-1), &\mbox{if~} n \mbox{~is odd};  \\
 		\tau(n)+(2\alpha+1)\tau(n')-2, & \mbox{if~} n=2^\alpha n', \mbox{with~} n' \mbox{~is odd.}
 	\end{cases}$

In particular, $\Gamma(D_n)$ is weakly perfect.
\end{enumerate}
 \end{cor}
\begin{proof}
\noindent{(i):} Proof follows by Corollary~\ref{permutability nonabelian c1} and Theorem~\ref{t37}.

 \noindent{(ii):} Since the number of proper normal subgroups of $D_n$ is $\tau(n)-1$ if $n$ is odd; $\tau(n)+1$ if $n$ is even, so the proof follows by
Theorem~\ref{t37} and Corollary~\ref{permutability nonabelian c1}.
\end{proof}
%
%
\begin{thm}\label{t66}Let $n\geq 3$ be an integer. Then
 $\Gamma(D_n)$ is a split graph if and only if $n$ is an odd prime.
\end{thm}
\begin{proof}
 By Theorem~\ref{t33}, $\Gamma_N(D_n)$ is a partite graph, so by definition of split graph and by Theorem~\ref{t35}, $\Gamma(D_n)$ is a split graph
if and only if $n$ is an odd prime.
\end{proof}

In the following result, we describe the structure of $\Gamma(D_n)$ for some values of $n$.
\begin{cor}\label{permutability nonabelian c8}
Let $n\geq 3$ be an integer.
\begin{enumerate}[{\normalfont (i)}]
\item If $n=2^\alpha$ and $\alpha\geq2$, then
$\Gamma(D_n)\cong K_{\alpha+2}+\Gamma_N(D_n)$,
where $\Gamma_N(D_n)$ is given by
\eqref{e71}.
\item If $n=p^\alpha$ where $p$ is  a prime, $p>2$, $\alpha\geq1$, then
$\Gamma(D_n)\cong K_\alpha+\Gamma_N(D_n)$,
where $\Gamma_N(D_n)$ is given by
\eqref{e72}.
\item If $n=2p$, where $p$ is  a prime, $p>2$, then
$\Gamma(D_n)\cong K_5+\Gamma_N(D_n)$,
where $\Gamma_N(D_n)$ is given by
\eqref{e73}.
\item If $n=pq$, where $p$, $q$ are distinct   primes, $2<p < q$, then
$\Gamma(D_n)\cong K_3+\Gamma_N(D_n)$,
where $\Gamma_N(D_n)$ is given by
\eqref{e74}.
\end{enumerate}
\end{cor}
\begin{proof}
Follows from Theorems~\ref{t11} and~\ref{permutability nonabelian t36}.
\end{proof}


%
 \begin{thm}\label{l1}
  Let $n\geq 3$ be an integer. If $\tau(n)+\sigma(n)<2(n+1)$, then $\Gamma(D_n)$ is non-Hamiltonian.
 \end{thm}
 \begin{proof}
Let $S=V(\Gamma(D_n))-\mathcal A_1$. Since $\tau(n)+\sigma(n)<2(n+1)$, so $|S|=\tau(n)+\sigma(n)-(n+2)< n=|\mathcal A_1|=c(\Gamma(D_n)-S)$. Thus by Theorem~\ref{l3}, $\Gamma(D_n)$ is non-Hamiltonian.
 \end{proof}

\begin{lemma}\label{l2}
Let $G_r$, $r=1,2,\ldots,m$ be vertex disjoint graphs  and $H$ be a graph with $n$ vertices.
Suppose that $G_r$, $r=1,2,\ldots,m$ and $H$ have  Hamiltonian paths. Let $G \cong \displaystyle H+\bigcup_{r=1}^m G_r$.
Then $G$ is Hamiltonian if and only if $m\leq n$.
\end{lemma}
\begin{proof}
For each $r=1,2,\ldots,m$, let $P_r$ be a Hamiltonian path in $G_r$ and let $P:v_1-v_2-\dots-v_n$ be a Hamiltonian path in $H$.
If $m\leq n$, then $v_1-P_2-v_2-P_3-v_3-\dots-v_{m-1}-P_m-v_m-v_{m+1}\dots-v_n-P_1-v_1$ is a Hamiltonian cycle in $G$.
If $m>n$, we take $S=V(H)$. Then $c(G-S)>|S|$ and so by Theorem~\ref{l3}, $G$ is non-Hamiltonian.
\end{proof}

\begin{thm}\label{permutability nonabelian t25}
\
\begin{enumerate}[{\normalfont (i)}]
\item If $n=p^\alpha$, where $p$ is prime, $\alpha\geq 1$, then $\Gamma(D_n)$ is Hamiltonian if and only if $p=2$ and $\alpha\geq 2$.
\item If $n=pq$, where $p$, $q$ are distinct primes, $p < q$, then $\Gamma(D_n)$ is Hamiltonian if and only if $p=2$ and $q\leq 5$.
\end{enumerate}
\end{thm}
\begin{proof}(i): We need to consider the following cases:

\noindent\textbf{Case a.} $p=2$.  By Corollary~\ref{permutability nonabelian c8}(i), $\displaystyle \Gamma(D_n)\cong H+\bigcup_{r=1}^2 G_r$,
where $H\cong K_{\alpha+2}$
and $G_r\cong \Gamma(D_{\frac{n}{2}})+K_2$. Now by Theorem~\ref{c1} and Lemma~\ref{l2}, $\Gamma(D_n)$ is Hamiltonian.

\noindent \textbf{Case b.} $p>2$.
Let $S=V(\Gamma(D_n))-\mathcal A_1$. Then $|S|=\tau(n)+\sigma(n)-(n+2)< n=|\mathcal A_1|=c(\Gamma(D_n)-S)$. So by Theorem~\ref{l3}, $\Gamma(D_n)$ is non-Hamiltonian.

%

\noindent (ii): We give the proof in the following cases:

\noindent \textbf{Case a.} $p=2$. Then by Corollary~\ref{permutability nonabelian c8}(iii),  $\Gamma(D_n)\cong H+\displaystyle \bigcup_{r=1}^q G_r$,
where $H\cong K_5$ and for each $r=1,2,\ldots,q$, $G_r\cong K_3$. Then by Lemma~\ref{l2}, $\Gamma(D_n)$ is Hamiltonian if and only if $q\leq 5$.

 \noindent \textbf{Case b.} $p\geq 3$ and $q\geq 5$.
 Let $S=V(\Gamma(D_n))-\mathcal A_1$. Then $|S|=\tau(n)+\sigma(n)-(n+4)< n=|\mathcal A_1|=c(\Gamma(D_n)-S)$. So by Theorem~\ref{l3}, $\Gamma(D_n)$ is non-Hamiltonian.
\end{proof}

\section{Quaternion groups}\label{sec:qua}
For any integer $n>1$, the quaternion group of order $4n$, is defined by
\begin{equation*}
Q_n=\langle a,b|a^{2n}=b^4=1,b^2=a^n,ab=ba^{-1}\rangle.
\end{equation*}
The subgroups of $Q_n$ are listed below:
\begin{itemize}
\item [(i)] cyclic groups $H_{0,r}=\langle a^{\frac{2n}{r}}\rangle$, of order $r$, where $r$ is a divisor of $2n$;
\item [(ii)]cyclic groups $H_{i,1}=\langle a^ib\rangle$ of order 4, where $i=1,\ldots,n$;
\item [(iii)] quaternion groups $H_{i,r}=\langle a^{\frac{n}{r}},a^ib\rangle$ of order $4r$, where $r$ is a divisor of $n$, $i=1,\ldots,\frac{n}{r}$.
\end{itemize}
\noindent The proper normal subgroups of $Q_n$ are the subgroups $H_0^r$, $r\neq 1$ listed in (i), when $n$ is odd; the subgroups $H_0^r$, $r\neq 1$
listed in (i) and $H_{i,\frac{n}{2}}$, $i=1,2$ of index 2, when $n$ is even. Thus
\begin{equation}\label{e61}
|L(Q_n)|=\tau(2n)+\sigma(n),
\end{equation}
\noindent and so
\begin{equation*}
|V(\Gamma(Q_n))|=\tau(2n)+\sigma(n)-2,
\end{equation*}
\begin{equation*}
|V(\Gamma_N(D_n)|=
\begin{cases}
		\sigma(n)-1, & \mbox{if~}n~\mbox{is}~\mbox{odd}; \\
		\sigma(n)-3, &   \mbox{otherwise.}
	\end{cases}
\end{equation*}
\noindent It is well known that $Z(Q_n)\cong \langle a^n\rangle$ is the  unique minimal subgroup of $Q_n$ and
\begin{equation*}
\frac{Q_n}{Z(Q_n)}\cong D_{n}.
\end{equation*}
\noindent For each positive divisors $r,s$ of $n$ and $i\in\{1,2\ldots,\frac{n}{r}\}$, $j\in \{1,2\ldots,\frac{n}{s}\}$, consider the
subgroups $H_{i,r}$ and $H_{j,s}$,
of $Q_n$.

$H_{i,r}$ and
$H_{j,s}$ permutes if and only if $\displaystyle\frac{H_{i,r}}{Z(Q_n)}$ and $\displaystyle \frac{H_{j,s}}{Z(Q_n)}$ permutes.

\noindent But
\begin{equation*}
\frac{H_{i,r}}{Z(Q_n)}\cong \langle x^{\frac{n}{r}},yx^{i'-1}\rangle\cong H_{i'}^r
\end{equation*}
\noindent and
\begin{equation*}
 \frac{H_{j,s}}{Z(Q_n)}\cong \langle x^{\frac{n}{s}},yx^{j'-1}\rangle\cong H_{j'}^s,
\end{equation*}
\noindent where
$x=a\langle a^n\rangle$, $y=b\langle a^n\rangle, i=q_1r+i'$, $j=q_2s+j'$, $0\leq i'<r$ and  $0\leq j'<s$ for some $q_1$, $q_2 \in \mathds{Z}$.

Thus $H_{i,r}$ and $H_{j,s}$ permutes if and only if $H_i^r$ and $H_j^s$ permutes.

 In view of these together with the necessary and sufficient condition for the permutability of dihedral subgroups discussed in Section~\ref{sec:dih}, we have the following:

 $H_{i,r}$ permutes with $H_{j,s}$ if and only if $x^{2(i'-j')}\in \langle x^{\frac{n}{[r,s]}}\rangle$, that is,
\begin{equation}\label{e8}
\frac{n}{[r,s]}~\Big|~2(i'-j').
\end{equation}
Let $r$ be a fixed positive divisor of $n$ and $i\in\{1,2,\ldots,\frac{n}{r}\}$. Suppose that $i=q'r+i'$, where $0\leq i'<r$, $q'\in \mathds{Z}$. Then
it is easy to see that  $x_{i'}^r$, the number of solution of \eqref{e8} is equal to $x_i^r$, the number of solution of \eqref{e11}.

In view of these facts, we have the following result.

\begin{thm}\label{permutability nonabelian t42}
Let $n\geq 3$ be an integer. Then $\Gamma_N(Q_n) \cong \Gamma_N(D_n)$.
\end{thm}

\begin{thm}\label{t12}
Let $n>1$ be an integer. Then
$\Gamma(Q_n)\cong K_r+\Gamma_N(D_n)$,
where $r=\tau(2n)+1$ if $n$ is even; $r=\tau(2n)-1$ if $n$ is odd.
\end{thm}
\begin{proof}
Since the number of normal subgroups of $Q_n$ is $\tau(2n)+1$ if $n$ is even; $\tau(2n)-1$ if $n$ is odd, so the proof follows by
Theorems~\ref{permutability nonabelian t1}, \ref{permutability nonabelian t42} and \ref{t33}.
\end{proof}

\begin{cor}\label{permutability nonabelian c25}Let $n>1$ be an integer.
\begin{enumerate}[{\normalfont (i)}]
\item
$\deg_{\Gamma(Q_n)}(H_{0,r})=\tau(2n)+\sigma(n)-3$, for every divisor $r$ of $2n$, $r\neq 2n$.
\item
$\deg_{\Gamma(Q_n)}(H_{i,r})=\tau(2n)+{x_i}^r-3$, for every divisor $r$ of $n$, $r\neq n$, $i\in\{1,2,\ldots,\frac{n}{r}\}$, where $x_i^r$ is
given by \eqref{e17} if $n$ is odd;
where $x_i^r$ is given by \eqref{e50} if $n=2^\alpha$, $\alpha\geq 1$; $x_i^r$ is given by \eqref{e13} if $n=2^{\alpha}n'$, with $n'$ is odd and $\alpha\geq 1$.
\end{enumerate}
\end{cor}
\begin{proof}
Follows by Theorems~\ref{permutability nonabelian t42}, \ref{permutability nonabelian t40}, Corollary~\ref{permutability nonabelian c1} and \eqref{e61}.
\end{proof}

\begin{cor}\label{permutability nonabelian c36}Let $n>1$ be an integer and $g$ denotes the arithmetic function given by~\eqref{e18}.
 \begin{enumerate}[{\normalfont (i)}]
 \item If $n$ is odd, then
\begin{equation*}
|E(\Gamma(Q_n))|=\displaystyle \frac {1}{2}\big\{g(n)-5\sigma(n)+6+\tau(2n)(\tau(2n)+2\sigma(n)-5)\big\}.
\end{equation*}
\item  If $n=2^\alpha$, $\alpha\geq 2$, then
\begin{equation*}
|E(\Gamma(Q_n))|=\displaystyle \frac{1}{2}\big\{2^{\alpha+1}(6\alpha-5)+\alpha^2-3\alpha+10\big\}.
\end{equation*}

\item  If $n=2^\alpha n'$, with $n'$ is odd, $\alpha \geq 1$, then
\begin{equation*}
|E(\Gamma(Q_n))|=\displaystyle  \frac{1}{2}\big\{[(\alpha-1)2^{\alpha+3}+9] g(n')-5\sigma(n)+\tau(2n)+6 +\tau(2n)[\tau(2n)+2\sigma(n)-6]\big\}.
\end{equation*}
 \end{enumerate}
\end{cor}
\begin{proof}
 Follows by Theorem~\ref{permutability nonabelian t42}, Corollary~\ref{permutability nonabelian c33}, \eqref{e61} and from the fact that the number of normal subgroups of $Q_n$ is $\tau(2n)+1$ if $n$ is even; $\tau(2n)-1$ if $n$ is odd.
\end{proof}

Now we characterize the values of $n$ for which $\Gamma(Q_n)$ is Eulerian.
\begin{cor}\label{permutability nonabelian c50}
Let $n> 1$ be an integer.
\begin{enumerate}[{\normalfont (i)}]
 \item If $n$ is odd, then $\Gamma(Q_n)$ is non-Eulerian.
\item If $n=2^\alpha$, $\alpha\geq 2$, then $\Gamma(Q_n)$ is Eulerian if and only if $\alpha$ is even.

\item  If $n=2^\alpha n'$, with $n'$ is odd, $n'=p_1^{\alpha_1}p_2^{\alpha_2}\ldots p_k^{\alpha_k}$,
where $p_i$'s are distinct primes and $\alpha$, $\alpha_i\geq 1$, then $\Gamma(Q_n)$ is Eulerian if and only if $\alpha$ is even and $\alpha_i$ is even, for every $i\in
\{1,2,\ldots,k\}$.
\end{enumerate}
\end{cor}
\begin{proof}
\noindent{(i)}: By Corollary~\ref{permutability nonabelian c25}(ii), $\deg_{\Gamma(Q_n)}(H_{i,1})$ is odd.
So proof follows.

\noindent{(ii)}: By Corollary~\ref{permutability nonabelian c25}(i), $\deg_{\Gamma(Q_n)}(H_{0,r})$
 is even if
and only if either $\tau(2n)$ is odd and $\sigma(n)$ is even or $\tau(2n)$ is even and $\sigma(n)$ is odd. By Note~\ref{n1}, $\sigma(n)$ is odd, so
we must have $\tau(2n)$ is even. By Note~\ref{n1}, $\tau(2n)$ is even if and only if $\alpha$ is even.
Now by Corollary~\ref{permutability nonabelian c25}(ii), $\deg_{\Gamma(Q_n)}(H_{i,r})$ is even if and only if $x_i^r$ is odd.
By the above argument $\tau(2n)$ is even and by \eqref{e50}, $x_i^r$ is
odd. So the proof follows.

\noindent{(iii)}. By Corollary~\ref{permutability nonabelian c25}(i), $\deg_{\Gamma(Q_n)}(H_{0,r})$
is even if
and only if either $\tau(2n)$ is even and  $\sigma(n)$ is odd or $\tau(2n)$ is odd and $\sigma(n)$ is even. Now we haver to consider following two cases.

\noindent\textbf{Case a.} $\alpha$ is odd. By Note~\ref{n1}, there is no such  $n$ exist.

\noindent\textbf{Case b.} $\alpha$ is even. By Note~\ref{n1}, $\tau(2n)$ is even, so we must have $\sigma(n)$ is odd. But $\sigma(n)$ is odd if and only if
$\sigma(n')$ is odd, by Note~\ref{n1}, $\sigma(n')$ is odd if and only if $\alpha_i$ is even for every $i\in\{1,2,\ldots,k\}$. Now
by Corollary~\ref{permutability nonabelian c25}(ii),
$\deg_{\Gamma(Q_n)}(H_{i,r})$ is even if and only if either $\tau(2n)$ is even and $x_i^r$ is odd or $\tau(2n)$ is odd and
$x_i^r$ is even. But by the above argument  $\tau(2n)$ is even, so we must have $x_i^r$ is odd. By \eqref{e13} $x_i^r$ is odd if and only if $\sigma(n')$
is odd. Thus by Note~\ref{n1}, $x_i^r$ is odd if and only if $\alpha_i$ is even for every $i\in\{1,2,\ldots,k\}$.
 \end{proof}

 \begin{cor}\label{permutability nonabelian c7}
Let $n> 1$ be an integer.

\begin{enumerate}[{\normalfont (i)}]
\item $\alpha(\Gamma(Q_n))=
 	\begin{cases}
 		n, & \mbox{~~if~ } n \mbox{~is odd};  \\
 		\frac{n}{2}, & \mbox{~~if~ } n \mbox{~is even}
 	\end{cases}$
\item $\omega(\Gamma(Q_n))=\chi(\Gamma(Q_n))=
 	\begin{cases}
 		\tau(2n)+\tau(n)-2, & \mbox{if~} n \mbox{~is odd};  \\
 		\tau(2n)+(2\alpha+1)\tau(n')-2, & \mbox{if~} n=2^\alpha n', \mbox{with~} n' \mbox{~is odd.}
  	\end{cases}$

In particular, $\Gamma(Q_n)$ is weakly perfect.
\end{enumerate}
\end{cor}

\begin{proof}
\noindent (i): Proof follows by Theorems~\ref{permutability nonabelian t42}, \ref{t37} and Corollary~\ref{permutability nonabelian c1}.

\noindent (ii): Since the number of proper normal subgroups of $Q_n$ is $\tau(2n)-1$ if $n$ is odd; $\tau(2n)+1$ if $n$ is even,  so the proof follows from
Theorems~\ref{permutability nonabelian t42}, \ref{t37} and Corollary~\ref{permutability nonabelian c1}.
\end{proof}

In the next result, we describe the structure of $\Gamma(Q_n)$ for some values of $n$.
\begin{thm}\label{permutability nonabelian c9}
Let $n>1$ be an integer.
\begin{enumerate}[{\normalfont (i)}]
\item If $n=2^\alpha$, $\alpha\geq 1$, then
\begin{equation*}
\Gamma(Q_n)\cong
	\begin{cases}
		K_4, & \mbox{if~ } \alpha = 1;   \\
		K_{\alpha+3}+\Gamma_N(D_n), & \mbox{otherwise}.
	\end{cases}
\end{equation*}
 where $\Gamma_N(D_n)$ is given by \eqref{e71}.
\item If $n=p^\alpha$, where $p$ is a prime, $p>2, \alpha\geq1$, then
$\Gamma(Q_n)\cong K_{2\alpha+1}+\Gamma_N(D_n)$,
where $\Gamma_N(D_n)$ given by~\eqref{e72}.
\item If $n=2p$, where $p$ is a prime, $p>2$, then
$\Gamma(Q_n)\cong K_7+\Gamma_N(D_n)$,
where $\Gamma_N(D_n)$ given by
\eqref{e73}.
\item If $n=pq$, where $p$, $q$ are primes, $2< p< q$, then
$\Gamma(Q_n)\cong K_7+\Gamma_N(D_n)$,
where $\Gamma_N(D_n)$ given by
\eqref{e74}.
\end{enumerate}
\end{thm}
\begin{proof}
 If $n=2$, then $Q_n$ is dedekind with 4 proper subgroups and so $\Gamma(Q_n) \cong K_4$.  The proofs of the remaining cases follows by Theorems~\ref{permutability nonabelian t42},  \ref{t12} and \ref{permutability nonabelian t36}.
\end{proof}

%

\begin{thm}\label{l8}
 Let $n>1$ be an integer. If $\tau(2n)+\sigma(n)<2(n+1)$, then $\Gamma(Q_n)$ is non-Hamiltonian.
\end{thm}
\begin{proof}
Let $S=V(\Gamma(Q_n))-\mathcal A_1$. Since $\tau(2n)+\sigma(n)<2(n+1)$, we have $|S|=\tau(2n)+\sigma(n)-{n+2}<n=|\mathcal A_1|=c(\Gamma(Q_n)-S)$. So by
Theorem~\ref{l3}, $\Gamma(Q_n)$ is non-Hamiltonian.
\end{proof}

\begin{thm}\label{permutability nonabelian t26}\
\begin{enumerate}[{\normalfont (i)}]
 \item If $n=p^\alpha$, where $p$ is a prime, $\alpha\geq1$, then $\Gamma(Q_n)$ is Hamiltonian if and only if
either $p=2$ and $\alpha\geq1$ or $p=3$ and $\alpha=1$.
\item If $n=pq$, where $p$, $q$ are distinct primes, then $\Gamma(Q_n)$ is Hamiltonian if and only if either $p=2$ and $q\leq 7$ or $p=3$ and $q=5$.
\end{enumerate}
\end{thm}
\begin{proof}
\noindent (i): Proof is divided into two cases.

\noindent\textbf{Case a.} $p=2$. By Theorem~\ref{permutability nonabelian c9}(i), $\Gamma(D_n)\cong H+G$, where $H\cong K_{\alpha+3}$
and $G\cong \Gamma_N(D_n)$. Now by Theorem~\ref{c1} and Lemma~\ref{l2}, $\Gamma(Q_n)$ is Hamiltonian.

\noindent \textbf{Case b.} $p>2$.

If $\alpha=1$, then by Theorem~\ref{permutability nonabelian c9}(ii), $\Gamma(Q_n)\cong H+\displaystyle \bigcup_{r=1}^p G_r$,
where $H\cong K_3$ and for each $r=1,2,\ldots,p$, $G_r\cong K_1$. So by Lemma~\ref{l2}, $\Gamma(Q_n)$ is Hamiltonian if and only if $p=3$.

If $\alpha\geq 2$, we take $S=V(\Gamma(Q_n))-\mathcal A_1$. Then $|S|=\tau(2n)+\sigma(n)-{n+2}<n=|\mathcal A_1|=c(\Gamma(Q_n)-S)$. So by Theorem~\ref{l3}, $\Gamma(Q_n)$ is non-Hamiltonian.


\noindent (ii): We deal with the following cases:

\noindent \textbf{Case a.} $p=2$. Then by Theorem~\ref{permutability nonabelian c9}(iii), $\Gamma(Q_n)\cong H+\displaystyle \bigcup_{r=1}^q G_r$,
where $H\cong K_7$ and for each $r=1,2,\ldots,q$, $G_r\cong K_3$. So by Lemma~\ref{l2},  $\Gamma(Q_n)$ is Hamiltonian if and only if $q\leq 7$.

\noindent \textbf{Case b.} $2 < p <q $. By Theorem~\ref{permutability nonabelian c9}(iv), $\Gamma(Q_n)\cong H+\Gamma_N(D_n)$,
 where $H\cong K_7$ and $\Gamma_N(D_n)\cong \displaystyle \bigcup_{i=1}^p\bigcup_{j=1}^q G_{ij}$, for each $i=1,2,\ldots,p$ and $j=1,2,\ldots,q$ $G_{ij}\cong K_3$ with vertex set $\{u_i,v_j,w_{ij}\}$.
Let the vertex set of $H$ be $\{N_i~|~i=1,2,\ldots,k\}$.

If $p=3$, $q=5$, then $\Gamma_N(D_{15})$ is shown in Figure 1. It is easy to see that,
$w_{11}-v_1-w_{21}-$ $u_2-w_{22}-v_2-$ $w_{12}-N_1-w_{23}-$ $v_3-w_{13}-N_2-$
$w_{24}-v_4-w_{14}-$ $N_3-w_{25}-v_5-$ $w_{35}-u_3-w_{32}-$ $N_4-w_{33}-N_5-$ $w_{34}-N_6-w_{31}-$ $N_7-w_{15}-u_1-w_{11}$ is a spanning cycle
in $\Gamma(Q_n)$.

If $p\geq 3$ and $q\geq 7$, we take $S=V(\Gamma(Q_n))-\mathcal A_1$. Then  $|S|=\tau(2n)+\sigma(n)-{n+2}<n=|\mathcal A_1|=c(\Gamma(Q_n)-S)$. So by Theorem~\ref{l3}, $\Gamma(Q_n)$ is non-Hamiltonian.
\end{proof}

\section{Quasi-dihedral groups}\label{sec:quasi}
For any positive integer $\alpha>3$, the quasi-dihedral group of order $2^\alpha$, is defined by
\begin{equation*}
QD_{2^\alpha}=\langle a,b~|~a^{2^{\alpha-1}}=b^2=1,
bab^{-1}=a^{2^{\alpha-2}-1}\rangle.
\end{equation*}
 The proper subgroups of $QD_{2^\alpha}$ are listed below:
\begin{itemize}
\item [(i)] cyclic groups $H_0^r=\langle a^{\frac{2^{\alpha-1}}{r}}\rangle$,  where $r$ is a divisor of $2^{\alpha-1}$, $r\neq 1$;
\item [(ii)]  the dihedral group $H_1^{2^{\alpha-2}}=\langle a^2$, $b\rangle\cong D_{2^{\alpha-2}}$ and the dihedral subgroups $H_i^r$ of $H_1^{2^{\alpha-2}}$,
where $r$ is a divisor of $2^{\alpha-2}$, $r\neq 2^{\alpha-2}$, $i=1,2,\ldots,\frac{2^{\alpha-2}}{r}$;
\item [(iii)]  the quaternion group $H_{2,2^{\alpha-3}}=\langle a^2,ab\rangle\cong Q_{2^{\alpha-3}}$ and the quaternion subgroups $H_{i,r}$
of $H_{2,2^{\alpha-3}}$, where $r$ is a divisor of $2^{\alpha-3}$, $r\neq 2^{\alpha-3}$, $i=1,2,\ldots,$ $\frac{2^{\alpha-3}}{r}$.
\end{itemize}
The only proper normal subgroups are listed in (i) together with $H_1^{2^{\alpha-2}}$, $H_{2,2^{\alpha-3}}$, $i=1,2$ of index 2.
Thus
\begin{equation}\label{e62}
 |L(QD_{2^{\alpha}})|=\alpha+3.2^{\alpha-2}-1,
\end{equation}
\noindent and so
\begin{equation*}
|V(\Gamma(QD_{2^\alpha}))|=\alpha+3(2^{\alpha-2}-1),
\end{equation*}
\begin{equation*}
 |V(\Gamma_N(QD_{2^\alpha}))|=3.2^{\alpha-3}-4.
 \end{equation*}

\subsection{Properties of $\Gamma_N(QD_{2^\alpha})$}
First we describe the structure of $\Gamma_N(QD_{2^\alpha})$.
\begin{thm}\label{permutability nonabelian t10}
Let $\alpha>3$ be an integer. Then
\begin{align}
 \Gamma_N(QD_{2^\alpha})\cong(K_2+\Gamma_N(D_{2^{\alpha-2}}))\cup(K_2+\Gamma_N(D_{2^{\alpha-3}})),
\end{align}\label{b1}
 where $\Gamma_N(D_{2^\alpha})$ is given by \eqref{e71}.
\end{thm}
\begin{proof}
 The only non-normal subgroups of $QD_{2^\alpha}$ other than $\langle a\rangle$ are the dihedral subgroups of $H_1^{2^{\alpha-2}}$ and
quaternion subgroups of $H_{2, 2^{\alpha-3}}$. Here no non-normal subgroup of $H_1^{2^{\alpha-2}}$ permutes with any non-normal subgroup of
$H_{2, 2^{\alpha-3}}$.
So $\Gamma_N(QD_{2^\alpha})$ is the disjoint union of $\mathcal{G}_1\cup \mathcal{G}_2$, where $\mathcal{G}_1$ is the subgraph of $\Gamma_N(QD_{2^\alpha})$
induced by the dihedral
subgroups of $H_1^{2^{\alpha-2}}$; and
$\mathcal{G}_2$ is the subgraph of $\Gamma_N(QD_{2^\alpha})$ induced by the quaternion subgroups of $H_{2, 2^{\alpha-3}}$.
Here $H_1^{2^{\alpha-3}}$, $H_3^{2^{\alpha-3}}$ are normal subgroups of $H_1^{2^{\alpha-2}}$, but are not normal in
$QD_{2^\alpha}$. Also
 $H_{2,2^{\alpha-4}}$, $H_{4, 2^{\alpha-4}}$ are normal subgroups of $H_{2,2^{\alpha-3}}$, but are not normal in
$QD_{2^{\alpha}}$. In view of these, it is easy to see that
\begin{equation*}
 \mathcal{G}_1\cong K_2+\Gamma_N(H_1^{2^{\alpha-2}})\cong K_2+\Gamma_N(D_{2^{\alpha-2}}),
\end{equation*}
 \noindent and by Theorem~\ref{permutability nonabelian t42},
\begin{equation*}
\mathcal{ G}_2\cong K_2+\Gamma_N(H_{2,{2^{\alpha-3}}})\cong K_2+\Gamma_N(Q_{2^{\alpha-3}})\cong K_2+\Gamma_N(D_{2^{\alpha-3}}).
\end{equation*}
\noindent Hence the proof.
\end{proof}

Next, we give the degrees of the vertices of $\Gamma_N(QD_{2^\alpha})$.
\begin{thm}\label{permutability nonabelian t43}
Let $\alpha> 3$ be an integer.
\begin{enumerate}[{\normalfont (i)}]
\item
 $\deg_{\Gamma_N(QD_{2^\alpha})}(H_i^r)=x_i^r-2$, for every divisor $r$ of $2^{\alpha-2}$, $r\neq 2^{\alpha-2}$, $i=1,2,\ldots,\frac{2^{\alpha-2}}{r}$,
where $x_i^r$ is given by~\eqref{e50}.

\item
$\deg_{\Gamma_N(QD_{2^\alpha})}(H_{i,r})=x_i^r-2$, for every divisor $r$ of $2^{\alpha-3}$, $r\neq 2^{\alpha-3}$, $i=1,2,\ldots,\frac{2^{\alpha-3}}{r}$,
where $x_i^r$ is given by~\eqref{e50}.
\end{enumerate}
\end{thm}
\begin{proof}
For every divisor $r$ of $2^{\alpha-2}$,  $r\neq 2^{\alpha-2}$, $2^{\alpha-3}$, $i=1,2,\ldots, \frac{2^{\alpha-2}}{r}$, by Theorems~\ref{permutability nonabelian t10} and \ref{permutability nonabelian t40}, we have
$\deg_{\Gamma_N(QD_{2^\alpha})}(H_i^r)=\deg_{\Gamma_N(D_{2^{\alpha-2}})}(H_i^r)+2$
$=x_i^r-4+2, =x_i^r-2$, where $x_i^r$ is  given in \eqref{e50}.

For $r=2^{\alpha-3}$, $i=1,2$, we have
$\deg_{\Gamma_N(QD_{2^\alpha})}(H_i^r)=x_i^r-|\{H_i^r, H_1^{2^{\alpha-2}}\}|
= x_i^r-2$, where $x_i^r$ is given in \eqref{e50}.

We can use the similar argument to prove the part (ii) of this result, in view of Theorem~\ref{permutability nonabelian t42}.
\end{proof}

\begin{cor}\label{permutability nonabelian c38}Let $\alpha>3$ be an integer. Then
\begin{equation*}
|E(\Gamma_N(QD_{2^\alpha}))|=2^{\alpha-3}(12\alpha-49)+14.
\end{equation*}
\end{cor}
\begin{proof}
By Theorem~\ref{permutability nonabelian t10},
\begin{align}\label{a1}
 |E(\Gamma_N(QD_{2^\alpha}))|&=|E(K_2+\Gamma_N(D_{2^{\alpha-2}}))|+|E(K_2+\Gamma_N(D_{2^{\alpha-3}}))|\nonumber \\
&=2|V(\Gamma_N(D_{2^{\alpha-2}}))|+1+|E(\Gamma_N(D_{2^{\alpha-2}}))|\nonumber\\
&\quad{}+2|V(\Gamma_N(D_{2^{\alpha-3}}))|+1+|E(\Gamma_N(D_{2^{\alpha-3}}))|.
\end{align}
Now by applying Corollary~\ref{permutability nonabelian c33} and \eqref{e110}, in \eqref{a1}, the result follows.
\end{proof}

\begin{cor}\label{permutability nonabelian c30}Let $\alpha>3$ be an integer.

\begin{enumerate}[{\normalfont (i)}]
 \item $\alpha(\Gamma_N(QD_{2^\alpha}))=3.2^{\alpha-4}$.
\item $\omega(\Gamma_N(QD_{2^\alpha}))=2(\alpha-2)$.
\item $\chi(\Gamma_N(QD_{2^\alpha}))=2(\alpha-2)$.
\item $\Gamma_N(QD_{2^\alpha})$ is weakly perfect.
\item $\gamma(\Gamma_N(QD_{2^\alpha}))= 2$
\item $\Gamma_N(QD_{2^\alpha})$ is non-Eulerian.
\item $\Gamma_N(QD_{2^\alpha})$ is  non-Hamiltonian.
\end{enumerate}
\end{cor}
\begin{proof}
{(i)}: Using Theorem~\ref{permutability nonabelian t10},
\begin{align*}
\alpha(\Gamma_N(QD_{2^\alpha}))&=\alpha(K_2+\Gamma_N(D_{2^{\alpha-2}}))+\alpha(K_2+\Gamma_N(D_{2^{\alpha-3}})) \\
&=\alpha(\Gamma_N(D_{2^{\alpha-2}}))+\alpha(\Gamma_N(D_{2^{\alpha-3}})).
\end{align*}
So the proof follows by Theorem~\ref{t37}(i).

\noindent {(ii)}: Using Theorem~\ref{permutability nonabelian t10},
\begin{align*}
\omega(\Gamma_N(QD_{2^\alpha}))&=\mbox{max}\{\omega(K_2+\Gamma_N(D_{2^{\alpha-2}})), \omega(K_2+\Gamma_N(D_{2^{\alpha-3}}))\} \\
&=\mbox{max}\{2+\omega(\Gamma_N(D_{2^{\alpha-2}})), 2+\omega(\Gamma_N(D_{2^{\alpha-3}}))\}.
\end{align*}
So the proof follows by Theorem~\ref{t37}(ii).

\noindent {(iii)}: Using Theorem~\ref{permutability nonabelian t10},
\begin{align*}
\chi(\Gamma_N(QD_{2^\alpha}))&=\mbox{max}~\{\chi(K_2+\Gamma_N(D_{2^{\alpha-2}})), \chi(K_2+\Gamma_N(D_{2^{\alpha-3}}))\} \\
&=\mbox{max}~\{2+\chi(\Gamma_N(D_{2^{\alpha-2}})), 2+\chi(\Gamma_N(D_{2^{\alpha-3}}))\}
\end{align*}
So the proof follows by Theorem~\ref{t37}(iii).

\noindent {(iv)}: Follows from (ii) and (iii).

\noindent {(v)}: Follows from Theorem~\ref{permutability nonabelian t10}.

\noindent {(vi)-(vii)}: Since by Theorem~\ref{permutability nonabelian t10}, $\Gamma_N(QD_{2^\alpha})$ is disconnected, so the proof follows.
\end{proof}

\subsection{Properties of $\Gamma(QD_{2^\alpha})$}
In the following result we describe the structure of $\Gamma(QD_{2^\alpha})$
\begin{thm}\label{permutability nonabelian c10}Let $\alpha>3$ be an integer. Then
$\Gamma(QD_{2^\alpha})\cong K_{\alpha+1}+\Gamma_N(QD_{2^\alpha})$, where $\Gamma_N(QD_{2^\alpha})$ is given by~\eqref{b1}
\end{thm}
\begin{proof}
Since the number of proper normal subgroups of $QD_{2^\alpha}$ is $\alpha+1$. So by Theorems~\ref{permutability nonabelian c1} and
~\ref{permutability nonabelian t10},
the proof follows.
\end{proof}

\begin{cor}\label{permutability nonabelian c26}
Let $\alpha>3$ be an integer.
\begin{enumerate}[{\normalfont (1)}]
\item $\deg_{\Gamma(QD_{2^\alpha})}(H_0^r)=\alpha+3.2^{\alpha-2}-4$, for every divisor $r$ of $2^{\alpha-1}$, $r\neq 1$.

\item $\deg_{\Gamma(QD_{2^\alpha})}(H_1^{2^{\alpha-2}})=\alpha+3.2^{\alpha-2}-4=\deg_{\Gamma(QD_{2^{\alpha}})}(H_2^{2^{\alpha-3}})$.

\item $\deg_{\Gamma(QD_{2^\alpha})}(H_i^r)=\alpha+x_i^r-1$, for every divisor $r$ of $2^{\alpha-2}$, $ r\neq 2^{\alpha-2}$, $i=1,2,\ldots,\frac{2^{\alpha-2}}{r}$,  where $x_i^r$ is
given by~\eqref{e50}.

\item   $\deg_{\Gamma(QD_{2^\alpha})}(H_{i,r})=\alpha+x_i^r-1$, for every divisor $r$ of $2^{\alpha-3}$, $r\neq 2^{\alpha-3}$, $i=1,2,\ldots,\frac{2^{\alpha-3}}{r}$, where $x_i^r$ is
given by~\eqref{e50}.
\end{enumerate}
\end{cor}
\begin{proof}
Follows by Corollary~\ref{permutability nonabelian c1}, Theorem~\ref{permutability nonabelian t43} and \eqref{e62} and from the fact that  $H_1^{2^{\alpha-2}}$ and
$H_2^{2^{\alpha-3}}$ are normal in $QD_{2^\alpha}$.
\end{proof}

\begin{cor}\label{permutability nonabelian c37}Let $\alpha>3$ be an integer. Then
\begin{equation*}
|E(\Gamma(QD_{2^\alpha}))|=\displaystyle \frac{1}{2}\big\{2^{\alpha-2}(18\alpha-43)+\alpha^2-7\alpha+20 \big\}.
\end{equation*}
\end{cor}
\begin{proof}
 Follows by Corollaries~\ref{permutability nonabelian c38}, \ref{permutability nonabelian c1}, \eqref{e62} and from the fact that the number of proper normal subgroups of $QD_{2^\alpha}$ is $\alpha+1$.
\end{proof}

\begin{cor}\label{permutability nonabelian c11}
Let $\alpha>3$ be an integer.

\begin{enumerate}[{\normalfont (i)}]
\item $\alpha(\Gamma(QD_{2^\alpha}))=3.2^{\alpha-4}$.
\item $\omega(\Gamma(QD_{2^\alpha}))=3(\alpha-1)$.
\item $\chi(\Gamma(QD_{2^\alpha}))=3(\alpha-1)$.
\item $\Gamma(QD_{2^\alpha})$ is weakly perfect.
\item $\Gamma(QD_{2^\alpha})$ is non-Eulerian.
\item $\Gamma(QD_{2^\alpha})$ is Hamiltonian.
\end{enumerate}
\end{cor}
\begin{proof}
(i)-(iii): Follows by the part (iii), (iv), and (v) of Corollary~\ref{permutability nonabelian c1} and by parts (i), (ii), (iii) of Corollary~\ref{permutability nonabelian c30}.

\noindent{(iv)}: Follows by \normalfont{(ii)} and \normalfont{(iii)}.

\noindent{(v)}: By Corollary~\ref{permutability nonabelian c26}\normalfont{(ii)}, if $\alpha$ is odd, then $\deg_{\Gamma(QD_{2^\alpha})}(H_i^r)$ is odd for
every divisor $r$ of $2^{\alpha-2}$, and by Corollary~\ref{permutability nonabelian c26}\normalfont{(i)}, if $\alpha$ is even, then $\deg(H_0^r)$ is odd for every divisor $r$ of $2^{\alpha-1}$. So the proof follows.

\noindent{(vi)}: By Theorems~\ref{permutability nonabelian t10} and \ref{permutability nonabelian c10},
$\Gamma(QD_{2^\alpha})=H+G_1\cup G_2$, where $H\cong K_{\alpha+1}$, $G_1\cong K_2+\Gamma_N(D_{2^{\alpha-2}})$ and $G_2\cong K_2+\Gamma_N(D_{2^{\alpha-3}})$.
Now, by
Corollary~\ref{c1}, $G_1$ and $G_2$ contains a Hamiltonian path. So by Lemma~\ref{l2}, $\Gamma(QD_{2^\alpha})$ is Hamiltonian.
\end{proof}

\section{Modular groups}\label{sec:mod}
For any integer $\alpha>2$ and any prime $p$, the modular group $M_{p^\alpha}$ of order $p^\alpha$ is defined by
\begin{equation*}
M_{p^\alpha}=\langle a,b|~a^{p^{\alpha-1}}= b^p=1, bab^{-1}=a^{p^{{\alpha-2}}-1}\rangle.
\end{equation*}
If $p^\alpha=8$, then $M_8\cong D_4$ and its corresponding permutability graphs are given by Theorem~\ref{permutability nonabelian t36}\normalfont{(i)}
and Corollary~\ref{permutability nonabelian c8}\normalfont{(i)}.
If $p^\alpha\neq 8$, then the subgroup lattice of $M_{p^\alpha}$ is isomorphic to that of $\mathds{Z}_{p^{\alpha -1}} \times \mathds{Z}_p$ and it is shown in~\cite[p. 210, Figure 4]{boh}.
Also if $p^\alpha\neq 8$, then the only proper normal subgroups of $M_{p^\alpha}$ are $\langle a^{p^{\alpha-2}}\rangle$, $\langle a^ib^j\rangle$ and $\langle a^k,b\rangle$, $i=1,2,\ldots,p^{\alpha-3}, j=0,1,\ldots,p-1$ and
$k=1,2,\ldots,p^{\alpha-2}$. Hence

\begin{equation*}\label{e63}
 |L(M_{p^\alpha})|=
	\begin{cases}
		4, & \mbox{if~} p^\alpha=8;  \\
		(\alpha-1)(p+1)+2, & \mbox{otherwise}.
	\end{cases}
\end{equation*}
\noindent and so
\begin{equation*}
|V(\Gamma(M_{p^\alpha}))|=
	\begin{cases}
		8, & \mbox{if~ } p^\alpha=8;  \\
		(\alpha-1)(p+1), & \mbox{otherwise}.
	\end{cases}
\end{equation*}
 \begin{equation*}
 |V(\Gamma_N(M_{p^\alpha}))|=
	\begin{cases}
		4, & \mbox{if~} p^\alpha=8;  \\
		p, & \mbox{otherwise}.
	\end{cases}
\end{equation*}
It is well known that if $p^\alpha \neq 8$, then any two subgroups of $M_{p^\alpha}$
permutes with each other.

Combining all these facts together, we have the following result.

\begin{thm}\label{permutability nonabelian t30}
Let $\alpha>3$ be an integer.

\begin{enumerate}[{\normalfont (i)}]
 \item $\Gamma_N(M_{p^\alpha})\cong
	\begin{cases}
		2K_2, & \mbox{if~ } p^\alpha =8;  \\
		K_p, & \mbox{otherwise.~ }
	\end{cases}$

\item $\Gamma(M_{p^\alpha})\cong
	\begin{cases}
		K_4+2K_2, & \mbox{if~ } p^\alpha =8;  \\
		K_{(\alpha-1)(p+1)}, & \mbox{otherwise.~ }
	\end{cases}$
\end{enumerate}
\end{thm}
Since  $\Gamma_N(M_{p^\alpha})$ and $\Gamma(M_{p^\alpha})$ are complete graphs if $p^\alpha\neq 8$, so one can easily obtain the other properties of
these graphs.

\section{Conclusions}
In this paper, we have studied the structure and some properties of permutability graphs of subgroups and permutability graph of non-normal subgroups
of the groups $D_n$,  $Q_n$,   $QD_{2^\alpha}$ and $M_{p^\alpha}$.
In particular, we showed that the structure of $\Gamma(D_n)$, $\Gamma_N(Q_n)$, $\Gamma(Q_n)$, $\Gamma_N(QD_{2^\alpha})$ and $\Gamma(QD_{2^\alpha})$
 depends on the structure of $\Gamma_N(D_n)$. In this sequel, in Theorems~\ref{permutability nonabelian t10} and \ref{permutability nonabelian c10}, we explicitly described structure of $\Gamma_N(QD_{2^\alpha})$ and $\Gamma(QD_{2^\alpha})$ respectively.
For the values of $n$ other than that mentioned in Theorem~\ref{permutability nonabelian t36}, the structure of the graph $\Gamma_N(D_n)$ becomes complicated and so further investigation  of the structure of $\Gamma_N(D_n)$ is essential to study the further properties of these remaining graphs.

Also, in Theorems~\ref{l7},~\ref{l1},~\ref{permutability nonabelian t25},~\ref{l8}, \ref{permutability nonabelian t26} and in Corollary \ref{permutability nonabelian c14}, we have discussed the Hamiltonicity of $\Gamma_N(D_n)$, $\Gamma(D_n)$, $\Gamma(Q_n)$ for some values of $n$. Now we pose the following:
\begin{prob}
Determine the values of $n$, for which $\Gamma_N(D_n)$, $\Gamma(D_n)$, $\Gamma(Q_n)$ are Hamiltonian.
\end{prob}



\begin{thebibliography}{33}


\bibitem{abdolla} A. Abdollahi, S. Akbari, H. R. Maimani, Non-commuting graph of a group, J. Algebra 28 (2006) 468-492.

\bibitem{abd} A. Abdollahi, A. Mohammadi Hassanabadi, Non-cyclic graph of a group, Comm. Algebra, 35 (2007) 2057-2081.

\bibitem{bertram} E. A. Bertram, M. Herzog, A. Mann, On a graph related to conjugacy classes
of groups, Bull. London Math. Soc, 22 (1990) 569-575.

\bibitem{binachi 1} M. Bianchi, A. Gillio, L. Verardi, Finite groups and subgroup-permutability,
 Ann. Mat. Pura Appl, 169 (1995) 251-268.
\bibitem{gillio} M. Bianchi, A. Gillio, L. Verardi, Subgroup-permutability and affine planes, Geometriae Dedicata, 85 (2001) 147-155.

\bibitem{boh} J. P. Bohanon, Les Reid, Finite groups with planar subgroup lattices, J. Algebra Comb, 23 (2006) 207-223.

\bibitem{binachi 2} A. Gillio, L. Verardi, On finite groups with a reducible permutability-graph, Ann. Mat. Pura Appl, 171 (1996) 275-291.

\bibitem{Marius} Marius T$\check{a}$rn$\check{a}$uceanu, Subgroup commutativity degrees of finite groups, J. Algebra, 321 (2009) 2508-2520.


\bibitem{mart} Martin J. Erickson, Anthony Vazzana, Introduction to Number Theory, Chapman and Hall/CRC., Florida, 2008.

\bibitem{raj1} R. Rajkumar, P. Devi, Planarity of permutability graphs of subgroups of groups, J. Algebra Appl, 13 (2014) Article No. 1350112.

\bibitem{raj2} R. Rajkumar, P. Devi, On permutability graphs of subgroups of groups, Discrete Math. Algorithm. Appl, accepted for publication. DOI: 10.1142/S1793830915500123

\bibitem{rob}  D. J. S. Robinson, A course in the theory of groups, Springer-Verlag, 1996.

\bibitem{will} J. S. Williams, Prime graph components of finite groups, J. Algebra, 69 (1981) 487-513.

\bibitem{west} D. B. West, Introduction to graph theory, Prentice-Hall, New Delhi, 2006.

\end{thebibliography}
\end{document}